\documentclass[11pt, a4paper]{amsart}
\usepackage{amsmath}
\usepackage{amssymb}
\usepackage{mathrsfs}
\usepackage[sort,numbers]{natbib}
\usepackage{color}
\usepackage[unicode=true,
 bookmarks=true,bookmarksnumbered=false,bookmarksopen=false,
 breaklinks=false,pdfborder={0 0 1},backref=section,colorlinks=true]
 {hyperref}
\hypersetup{
 linkcolor=blue, urlcolor=marineblue, citecolor=green, pdfstartview={FitH}, hyperfootnotes=true, unicode=true}
\numberwithin{equation}{section}
 
\newtheorem{theorem}{Theorem}[section]
\newtheorem{proposition}[theorem]{Proposition}

\newtheorem{lemma}[theorem]{Lemma}
\newtheorem{corollary}[theorem]{Corollary}

\newtheorem{remark}[theorem]{Remark}

\begin{document}

\baselineskip=16pt

\thispagestyle{empty}

\begin{center}\sf
{\Large $L^p$-$L^q$ decay estimates for dissipative linear}\vskip.2cm
{\Large hyperbolic systems in 1D}\vskip.25cm
\today\vskip.2cm
Corrado MASCIA\footnote{Dipartimento di  Matematica ``G. Castelnuovo'',
	Sapienza -- Universit\`a di Roma, P.le Aldo Moro, 2 - 00185 Roma (ITALY), \texttt{\tiny mascia@mat.uniroma1.it}}
and
Thinh Tien NGUYEN\footnote{Department of Mathematics, Gran Sasso Science Institute -- Istituto Nazionale di Fisica Nucleare,
Viale Francesco Crispi, 7 - 67100 L'Aquila (ITALY),  \texttt{\tiny nguyen.tienthinh@gssi.infn.it}}
\end{center}
\vskip.5cm

\begin{quote}\footnotesize\baselineskip 12pt 
{\sf Abstract.}  
Given $A,B\in M_n(\mathbb R)$, we consider the Cauchy problem for partially dissipative hyperbolic systems
having the form
\begin{equation*}
	\partial_{t}u+A\partial_{x}u+Bu=0,
\end{equation*}
with the aim of providing a detailed description of the large-time behavior.
Sharp $L^p$-$L^q$ estimates are established for the distance between the solution to the system and a time-asymptotic profile, where the profile is the superposition of diffusion waves and exponentially decaying waves.
\vskip.15cm

\vskip.15cm

\end{quote}

\tableofcontents

\section{Introduction}
This work is concerned with the Cauchy problem
\begin{equation}\label{eq:original hyperbolic system}
	\partial_t u+\mathcal Lu:=\partial_t u+A\partial_xu+Bu=0,\qquad u(0,x)=u_0(x).
\end{equation}

Even if we are going to discuss properties of the system \eqref{eq:original hyperbolic system} on its own,
we primarily regard at the system \eqref{eq:original hyperbolic system} as a linearization of the nonlinear hyperbolic system
\begin{equation}\label{nonlinear1}
	\partial_t u+\partial_{x} F(u) + G(u) = 0
\end{equation}
at a constant stationary state $\bar u$ satisfying $G(\bar u)=0$ (see \citep{liu87,serre07} and descendants).
In addition, linear systems fitting in the class \eqref{eq:original hyperbolic system} emerge as models for velocity jump processes such as the Goldstein--Kac model \citep{goldstein51,kac74} (a generalization was introduced in \citep{mascia16}) and in other fields of application.

On the other hand, decay estimates for the system \eqref{eq:original hyperbolic system} have been accomplished for years. In \citep{shizuta85}, it is proved that the $L^2$-norm of the solution $u$ to \eqref{eq:original hyperbolic system} is bounded by the sum of the two terms: the first term, with respect to the $L^2$-norm of the initial datum of $u$, decays exponentially and the second one, with respect to the $L^q$-norm of the initial datum of $u$ for $q\in[1,\infty]$, decays at the rate $(1/q-1/2)/2$. In that work, the matrices $A$ and $B$ are symmetric and they satisfy the Kawashima--Shizuta condition: {\it if $z$ is an eigenvector of $A$, then $z$ does not belong to $\ker B$}, which is required for designing a compensating matrix to capture the dissipation of the system \eqref{eq:original hyperbolic system} over the degenerate kernel space of $B$ since the symmetric structure is not enough to guarantee the decay. The result is then improved in \citep{bianchini07}, if \eqref{eq:original hyperbolic system} has a convex entropy and satisfies the Kawashima--Shizuta condition, $B$ can be written in the block-diagonal form $\textrm{\normalfont diag}\,(O_{m\times m},D)$ where $O_{m\times m}$ is the $m\times m$ null matrix and $D\in M_{n-m}(\mathbb R)$ is positive definite, and by considering the parabolic equation given by applying the Chapman--Enskog expansion to the system \eqref{eq:original hyperbolic system}, the $L^p$-norm of the difference between the solution $u$ and the solution $U$ to the parabolic equation decays $1/2$ faster than the rate $(1-1/p)/2$ in terms of the $L^1\cap L^2$-norm of the initial datum of $u$. Recently, \citep{ueda12} can be seen as a generalization of \citep{shizuta85} for any non symmetric matrix $B$ under appropriate conditions.

More detailed descriptions of the asymptotic behavior have been provided for specific classes of equations by {\it $L^p$-$L^q$ estimates} {e.g.} the $L^p$-$L^q$ estimate for the Cauchy problem for the damped wave equation
\begin{equation}\label{dampedwave}
	\partial_t u+\partial_{tt} u-\Delta u=0.
\end{equation}
It shows that the time-asymptotic profile of the solution $u$ to \eqref{dampedwave} includes the solution to a heat equation and the solution to a wave equation, and when measuring the initial datum of $u$ in $L^q$, the $L^p$ distance between $u$ and this profile decays $\varepsilon>0$ faster than the rate $\alpha(p,q):=d(1/q-1/p)/2$ where $d$ is the spacial dimension (see \citep{marcati03,hosono04}).

We are not aware of any results on such kind of estimate for the general system \eqref{eq:original hyperbolic system} with general exponents $p$ and $q$. We start with the following assumptions on the matrices $A$ and $B$.

{\it {\bf Condition A.} {\sl [Hyperbolicity]} The matrix $A$ is diagonalizable with real eigenvalues.}

{\it {\bf Condition B.} {\sl [Partial dissipativity]} $0$ is a semi-simple eigenvalue of $B$ with algebraic multiplicity $m\geq 1$
and the spectrum $\sigma(B)$ of $B$ can be decomposed as $\{0\}\cup\sigma_{0}$
with $\sigma_0\subset\{\lambda\in\mathbb C\,:\,\textrm{\normalfont Re}\,\lambda>0\}$.}

Let $P_0^{(0)}$ be the unique eigenprojection associated with the eigenvalue $0$ of $B$, then the reduced system is given by
\begin{equation}\label{eq:reduced system}
	\partial_tw+C\partial_xw\approx 0,
\end{equation}
where $w:=P_0^{(0)}u$ and $C:=P_0^{(0)}AP_0^{(0)}$. One assumes

{\it {\bf Condition C.} {\sl [Reduced hyperbolicity]} The matrix $C$ is diagonalizable with real eigenvalues considered in $\ker(B)$.}

On the other hand, the requisite condition for the decay of the solution to the system \eqref{eq:original hyperbolic system}, which is related to the well-known Kawashima--Shizuta condition, is that the eigenvalues $\lambda(i\xi)$ of the operator $E(i\xi):=-(B+i\xi A)$ satisfy

{\it {\bf Condition D.} {\sl [Uniform dissipativity]} There is $\theta>0$ such that
\begin{equation}\label{eq:regularity}
	\textrm{\normalfont Re}\,(\lambda(i\xi))\le -\dfrac{\theta |\xi|^{2}}{1+|\xi|^2}, \qquad \textrm{ for } \xi\ne 0.
\end{equation}}

In this framework, we show that under the assumptions {\bf A}, {\bf B}, {\bf C} and {\bf D}, the time-asymptotic profile of the solution to the system \eqref{eq:original hyperbolic system} is the superposition of diffusion waves and exponentially decaying waves.

The diffusion waves are constructed as follows. Let $\Gamma_0$ be an oriented closed curve enclosing the eigenvalue $0$ except for the nonzero eigenvalues of $B$ in the resolvent set $\rho(B)$, one sets
\begin{equation}\label{eq:reduced resolvent coefficient}
	S_0^{(0)}:=\dfrac{1}{2\pi i}\int_{\Gamma_0}z^{-1}(B-zI)^{-1}\,dz.
\end{equation}
On the other hand, let
\begin{equation}\label{eq:P01}
	P_0^{(1)}:=-P_0^{(0)}AS_0^{(0)}-S_0^{(0)}AP_0^{(0)},
\end{equation}
one defines
\begin{equation}\label{eq:matrix D}
	D:=-\bigl(P_0^{(1)}BP_0^{(1)}+P_0^{(0)}AP_0^{(1)}+P_0^{(1)}AP_0^{(0)}\bigr).
\end{equation}
Then, we consider the Cauchy problem with respect to $U$ in $\textrm{\normalfont ran}\bigl(P_j^{(0)}\bigr)$, such that
\begin{equation}\label{eq:parabolic system}
	\partial_t U+c_j\partial_x U-P_j^{(0)}D \partial_{xx}U=0,\qquad U(0,x)=P_j^{(0)}u_0(x),
\end{equation}
where $c_j$ and $P_j^{(0)}$ are the $j$-th element of the spectrum of $C$ considered in $\ker(B)$ and the eigenprojection associated with it for $j=1,\dots,h$ and $h\le m$ is the cardinality of the spectrum of $C$ considered in $\ker(B)$ and $m$ is the algebraic multiplicity of the eigenvalue $0$ of $B$. Thus, one can choose $U:=\sum_{j=1}^{h}U_j$ where $U_j$ is the solution to the system \eqref{eq:parabolic system} for $j=1,\dots,h$.

On the other hand, the coefficients $P_0^{(0)}$ and $S_0^{(0)}$ can be computed by the formula
\begin{equation}
	P_0^{(0)}=\mathbb P_{m-1}(B)\qquad \textrm{and}\qquad S_0^{(0)}=\mathbb S_{m-1}(B),
\end{equation}
where the matrix-valued functions $\mathbb P$ and $\mathbb S$ are introduced in \eqref{func:P} and \eqref{func:S} in the appendix section. Moreover, let $\alpha>\max\{|\lambda|:\lambda\in\sigma(C)\}$ and let $C':=C+\alpha P_0^{(0)}$, then $C'$ has $h$ distinct nonzero eigenvalues denoted by $c_j'$ with algebraic multiplicities $m_j\ge 1$ for $j\in\{1,\dots,h\}$. Thus, $P_j^{(0)}$ can be computed by the formula
\begin{equation}
	P_j^{(0)}=\mathbb P_{m_j-1}(C'-c_j'I),
\end{equation}
for $j\in\{1,\dots,h\}$. Noting that the shift from $C$ to $C'$ is requisite since we consider only the eigenvalues of $C$ restricted to $\ker(B)$.

The exponentially decaying waves are constructed as follows. Due to the diagonalizable property of $A$, let $Q\in M_n(\mathbb R)$ be the invertible matrix diagonalizing $A$. Then, one sets
\begin{equation}\label{eq:A and B bar}
	\bar A:=\textrm{\normalfont diag}\,(a_1,\dots,a_n),\qquad \bar B:=Q^{-1}BQ,
\end{equation}
where $a_j\in\mathbb R$ for $j=1,\dots,n$ are the repeated eigenvalues of $A$. Let define a partition denoted by $\{\mathcal S_j:j=1,\dots,s\}$ of $\{1,\dots,n\}$ for some $s\le n$ such that $h,k\in\mathcal S_j$ if $a_h=a_k$, it is easy to see that $s$ is the cardinality of the spectrum of $\bar A$. On the other hand, we also define the matrix
\begin{equation}\label{eq:the eigenprojection associated with alphaj}
	\bigl(\Pi_j^{(0)}\bigr)_{hk}:=\begin{cases} 1 &\textrm{if } h=k\in\mathcal S_j,\\
	0 &\textrm{otherwise,} \end{cases}
\end{equation}
for $h,k=1,\dots,n$. Then we consider the Cauchy problem with respect to $V\in\textrm{\normalfont ran}\bigl(\Pi_j^{(0)}\bigr)$ such that
\begin{equation}\label{eq:hyperbolic system}
	\partial_t V+\alpha_j\partial_xV+\Pi_j^{(0)}\bar B V=0,\qquad V(0,x)=\Pi_j^{(0)}Q^{-1}u_0(x),
\end{equation}
where $\alpha_j=a_h$ if $h\in \mathcal S_j$ for $j=1,\dots,s$. Thus, we can choose $V:=Q\sum_{j=1}^sV_j$ where $V_j$ is the solution to the system \eqref{eq:hyperbolic system} for $j=1,\dots, s$.

\begin{theorem}[$L^p$-$L^q$]\label{theo:standard type}
If $u$ is the solution to the system \eqref{eq:original hyperbolic system} with the initial datum $u_0\in L^q(\mathbb R)$, the conditions {\bf A}, {\bf B}, {\bf C} and {\bf D} imply that, for $1\le q\le p\le \infty$ and $t\ge 1$, there are constants $C:=C(p,q)>0$ and $\delta>0$ such that
\begin{equation}\label{est:ellepq standard type}
	\|u-U-V\|_{L^p}\le Ct^{-\frac12\bigl(\frac1q-\frac1p\bigr)-\frac12}\|u_0\|_{L^q},
\end{equation}
where
\begin{equation}\label{est:parabolic solution-hyperbolic solution}
	\|U\|_{L^p}\le Ct^{-\frac12\bigl(\frac1q-\frac1p\bigr)}\|u_0\|_{L^q}\qquad \textrm{and}\qquad \|V\|_{L^2}\le Ce^{-\delta t}\|u_0\|_{L^2}.
\end{equation}
\end{theorem}

Going back to the $L^p$-$L^q$ decay estimate in \citep{marcati03}, if the initial condition for the Cauchy problem for \eqref{dampedwave} is given by $(u,\partial_t u)|_{t=0}=(u_0,u_1)$, then the following estimate holds
\begin{equation}\label{mnLpLq}
	\bigl\| u-U-e^{-t/2}V \bigr\|_{L^p}\le C\,t^{-\frac12(\frac1q-\frac1p)-1}\bigl\|u_0+u_1\bigr\|_{L^q}, \qquad \forall t\ge 1,
\end{equation}
where respectively, $U$ and $V$ are the solutions to the Cauchy problems
\begin{equation*}
	\left\{\begin{aligned}
		&\partial_t U-\partial_{xx} U=0,\\
		&U(x,0)=u_0(x)+u_1(x),
	\end{aligned}\right.
	\quad\textrm{and}\quad 
	\left\{\begin{aligned}
		&\partial_{tt}V-\partial_{xx}V=0,\\
		&V(x,0)=u_0(x),\quad \partial_tV(x,0)=0.
	\end{aligned}\right.
\end{equation*}
Comparing \eqref{est:ellepq standard type} with \eqref{mnLpLq}, we recognize a difference of 1/2 in the decay rates. The better decay, which is valid for the linear damped wave equation, is a consequence of an additional property, namely the invariance with respect to the transformation $x\mapsto -x$. Indeed, in terms of the Goldstein--Kac system, such symmetry implies that the eigenvalue curves of $E(i\xi)=-(B+i\xi A)$ which pass through $0$ can be expanded as $\lambda(i\xi):=-d_0\xi^2+\mathcal O(|\xi|^4)$ for some $d_0>0$ as $|\xi|\to 0$, and the fact that the error terms are $\mathcal O(|\xi|^4)$ guarantees the gain of $1$ instead of $1/2$ in the decay rate, where $1/2$ holds for general cases where the error terms are $\mathcal O(|\xi|^3)$.

Thus, we are also interested in systems fitting in the class \eqref{eq:original hyperbolic system} that have an analogous property, namely

{\bf Hypothesis S.} {\sl [Symmetry]} There is an invertible symmetric matrix $S$ such that
\begin{equation*}
	AS=-SA\qquad\textrm{and}\qquad BS=SB.
\end{equation*}

When the above assumption holds, if $u:=u(x,t)$ is a solution to \eqref{eq:original hyperbolic system}, then the reflection $v:=v(x,t)=u(-x,t)$ is a solution to the same system as well.

Let us consider a stronger assumption than the condition {\bf C} on the reduced system, namely

{\it {\bf Condition C'.} {\sl [Reduced strictly hyperbolicity]} The matrix $C$ is diagonalizable with $m$ real distinct eigenvalues considered in $\ker(B)$.}

Let $U=\sum_{j=1}^mU_j$ where $U_j$ is the solution to \eqref{eq:parabolic system} with the initial datum given by
\begin{equation}\label{eq:initial datum of parabolic system}
	U(0,x):=\bigr(P_j^{(0)}+P_j^{(1)}\partial_x\bigl)u_0(x),
\end{equation}
where $P_j^{(0)}$ is already introduced and $P_j^{(1)}$ is as follows for $j\in\{1,\dots,m\}$.

 Let $\Gamma_j$ be an oriented closed curve enclosing the nonzero eigenvalue $c_j'$ except for the other eigenvalues of $C'$ in the resolvent set $\rho(C')$ for $j\in\{1,\dots,m\}$. One sets
\begin{equation}
	S_j^{(0)}:=\dfrac{1}{2\pi i}\int_{\Gamma_j}z^{-1}(C'-zI)^{-1}\,dz,
\end{equation}
and then, $P_j^{(1)}$ can be computed by
\begin{equation}\label{eq:Pj1}
	P_j^{(1)}:=P_j^{(0)}DS_j^{(0)}+S_j^{(0)}DP_j^{(0)},
\end{equation}
for all $j\in\{1,\dots,m\}$. Similarly to before, $S_j^{(0)}$ can be computed by
\begin{equation}
	S_j^{(0)}=\mathbb S_{m_j-1}(C'-c_j'I),
\end{equation}
since $c_j'$ is simple for $j\in\{1,\dots,m\}$.

Let $V$ be the same as before, one has
\begin{theorem}[Increased decay rate]\label{theo:standard type symmetry}
With the same hypotheses as in Theorem \ref{theo:standard type}, if the condition {\bf C} is substituted by the condition {\bf C'} and if the condition {\bf S} holds, then, for $1\le q\le p\le \infty$, there is a positive constant $C:=C(p,q)>0$ such that
\begin{equation}\label{est:ellepq symmetry}
	\bigl\|u-U-V\bigr\|_{L^p}\le C\,t^{-\frac12\bigl(\frac1q-\frac1p\bigr)-1}
	\bigl\|u_0\bigr\|_{L^q},\qquad\forall t\ge1,
\end{equation}
where
\begin{equation}
	\|U\|_{L^p}\le Ct^{-\frac12\bigl(\frac1q-\frac1p\bigr)}\|u_0\|_{L^q}\qquad \textrm{and}\qquad \|V\|_{L^2}\le Ce^{-\delta t}\|u_0\|_{L^2}.
\end{equation}
\end{theorem}

Once relaxing from {\bf C'} to {\bf C}, the decay rate in the estimate \eqref{est:ellepq standard type} does not increase in general since the condition {\bf S} cannot prevent the eigenvalues of $E$ which converge to $0$ from exhibiting non zero terms $(i\xi)^{3+\alpha}$ for $\alpha\in [0,1)$ in their expansions, and thus, it does not permit to have the gain of 1 in the decay rate.

The paper is organized as follows. In order to study the behavior of the solution to the system \eqref{eq:original hyperbolic system}, we introduce the asymptotic expansion of the operator $E(i\xi)=-(B+i\xi A)$ in Section \ref{sec:asymptotic expansion}. Then, Section \ref{sec:fundamental solution} and Section \ref{sec:multiplier estimates} are devoted to give the {\it a priori} estimates of the solution to the system \eqref{eq:original hyperbolic system}. Moreover, the symmetry property of the system \eqref{eq:original hyperbolic system} is also discussed in Section \ref{sec:symmetry}. Then, we prove the main theorems in Section \ref{sec:proofs of main theorems}. Finally, we let the appendix section for some useful facts of the perturbation theory for linear operators in finite dimensional space together with a tool for computing the eigenprojections.

\subsection*{Notations and Definitions}
Given a matrix operator $A$, we denote $\ker(A)$, $\textrm{\normalfont ran}(A)$, $\rho(A)$ and $\sigma(A)$ the kernel, the range, the resolvent set and the spectrum of $A$ respectively.

On the other hand, we call $\lambda\in\mathbb C$ is an eigenvalue of $A$ considered in a domain $\mathcal D$ if there is $u\in \mathcal D$ such that $u\ne O_{n\times 1}$ and $Au=\lambda u$.

For $x\in \mathbb C$ small enough, if $A(x)=A^{(0)}+\mathcal O(|x|)$ and $\lambda\in\sigma(A)$ satisfying $\lambda(x)\to \lambda^{(0)}$ as $|x|\to 0$ where $\lambda^{(0)}\in\sigma(A^{(0)})$, the set of all such eigenvalues of $A$ is called the $\lambda^{(0)}$-group. Moreover, $P$ is called the total projection of a group if $P$ is the sum of the eigenprojections associated with the eigenvalues belonging to that group.

Let $T:\mathbb R\to \mathcal B$ where $\mathcal B$ is a Banach space with some suitable norm $|\cdot|_{\mathcal B}$. Define the $L^p(\mathbb R,\mathcal B)$-norm of $T$ as follows.
\begin{equation*}
	\|T\|_{L^p}:=\left(\int_{-\infty}^{+\infty}|T(x)|_{\mathcal B}\,dx\right)^{1/p},\qquad 1\le p<\infty,
\end{equation*}
and
\begin{equation*}
	\|T\|_{L^\infty}:=\textrm{ess\,sup}_{-\infty<x<+\infty} |T(x)|_{\mathcal B}.
\end{equation*}
From here, we use the notation $|\cdot|$ instead of $|\cdot|_{\mathcal B}$ to indicate the norm associated with $\mathcal B$.

Let $m$ be a tempered distribution, $m$ is called a Fourier multiplier on $L^p$, for $1\le p\le \infty$, if
\begin{equation*}
	\sup_{\|f\|_{L^p}=1}\|\mathcal F^{-1}(m)*f\|_{L^p}<+\infty.
\end{equation*}
The $M_p$ space, for $1\le p\le \infty$, is the space of Fourier multipliers endowed with the norm
\begin{equation*}
	\|m\|_{M_p}=\sup_{\|f\|_{L^p}=1}\|\mathcal F^{-1}(m)*f\|_{L^p}.
\end{equation*}

\section{Asymptotic expansions}\label{sec:asymptotic expansion}

We study the asymptotic expansions of the eigenvalues of the operator $E(i\xi)=-(B+i\xi A)$ by dividing the frequency domain $\xi\in \mathbb R$ into the low frequency as $|\xi |\to 0$, the intermediate frequency as $|\xi|$ away from $0$ and $+\infty$ and the high frequency as $|\xi|\to +\infty$.

Primarily, we consider the low-frequency case. Due to the fact that the eigenvalues of $E$ converge to the eigenvalues of $B$ as $|\xi|\to 0$ in general and the condition {\bf B}, the eigenvalues of $E$ are divided into two groups such that one among them contains the eigenvalues of $E$ converging to $0$ as $|\xi|\to 0$. Thus, we will study these two groups separately for the low-frequency case. We also recall the matrices $C$ and $D$ in \eqref{eq:reduced system} and \eqref{eq:matrix D} respectively.

\begin{proposition}[Low frequency 1]\label{prop:low frequency 1}
Let $h\in \mathbb Z^+$ be the cardinality of the spectrum of the matrix $C$ considered in $\ker(B)$. If the condition {\bf C} holds, then, for $j\in\{1,\dots,h\}$, there is $h_j\in \mathbb Z^+$ to be less than or equal to the algebraic multiplicity of the $j$-th eigenvalue of $C$ considered in $\ker(B)$ such that there are $h_j$ groups of the eigenvalues of $E$ and the approximation of the elements of the $\ell$-th group has the form
\begin{equation}\label{eq:expansion of lambdajell}
	\lambda_{j\ell}(i\xi)=-ic_j\xi-d_{j\ell} \xi^2+{\scriptstyle\mathcal O}(|\xi|^2),\qquad |\xi|\to 0,
\end{equation}
where $c_j\in\sigma(C)$ considered in $\ker(B)$ and $d_{j\ell}\in\sigma\bigl(P_j^{(0)}DP_j^{(0)}\bigr)$ considered in $\ker(C-c_jI)$ for $\ell=1,\dots,h_j$ with $P_j^{(0)}$ the eigenprojection associated with $c_j$. In particular, if the condition {\bf D} holds, then
\begin{equation}\label{eq:real part of djell}
	\textrm{\normalfont Re}\,(d_{j\ell})\ge \theta>0, \qquad \textrm{ for } \ell=1,\dots,h_j \textrm{ and } j=1,\dots, h.
\end{equation}

Moreover, the total projection associated with the $\ell$-th group is then approximated by
\begin{equation}\label{eq:expansion of Pjell}
	P_{j\ell}(i\xi)=P_{j\ell}^{(0)}+\mathcal O(|\xi|),\qquad |\xi|\to 0,
\end{equation}
where $P_{j\ell}^{(0)}$ is the eigenprojection associated with $d_{j\ell}$ considered in $\ker(C-c_jI)$ for $\ell\in\{1,\dots,h_j\}$ and $j\in\{1,\dots,h\}$.
\end{proposition}

\begin{proof}
This proof is dealt with the $0$-group of $E$ {\it i.e.} the group contains the eigenvalues of $E$ converging to $0$ as $|\xi|\to 0$. On the other hand, we can consider $T(\zeta):=B+\zeta A$ where $\zeta=i\xi$ instead of $E$ in order to apply Proposition \ref{prop:subprojections} and Proposition \ref{prop:construction of subprojections} since $E=-T$. The proof then includes three steps of approximation and reduction steps interlacing them. 

{\bf Step 0:} It is obvious that the approximation of the elements of the $0$-group of $T$ has the form
\begin{equation*}
	\lambda_0(\zeta)={\scriptstyle\mathcal O}(1),\qquad |\zeta|\to 0.
\end{equation*}
On the other hand, by Proposition \ref{prop:construction of subprojections}, the total projection associated with this group is approximated by
\begin{equation*}
	P_0(\zeta)=P_0^{(0)}+\mathcal O(|\zeta|),\qquad |\zeta|\to 0,
\end{equation*}
where $P_0^{(0)}$ is the eigenprojection associated with the eigenvalue $0$ of $B$.

In particular, we can perform a more accurate expansion of $P_0$. Indeed, we have
\begin{equation}\label{eq:expansion of P0}
	P_0(\zeta)=P_0^{(0)}+\zeta P_0^{(1)}+\mathcal O(|\zeta|^2),\qquad |\zeta|\to 0,
\end{equation}
where $P_0^{(1)}$ can be computed by the formula \eqref{eq:P01}. We will prove the formula \eqref{eq:P01} in brief. As $|\zeta|\to 0$, for $z\in \Gamma$ any compact set contained in the resolvent set $\rho(B)$ of $B$, we have the uniformly convergent expansion
\begin{equation}\label{eq:expansion of the resolvent of T}
	(T(\zeta)-zI)^{-1}= (B-zI)^{-1}-\zeta (B-zI)^{-1}A(B-zI)^{-1}+\mathcal O(|\zeta|),\qquad |\zeta|\to 0,
\end{equation}
and we also have the expansion about $0$ of the resolvent 
\begin{equation}\label{eq:expansion of the resolvent of B}
	(B-zI)^{-1}=\sum_{h=-\infty}^{-1}\bigl(N_0^{(0)}\bigr)^hz^h+P_0^{(0)}+\sum_{h=1}^{+\infty}\bigl(S_0^{(0)}\bigr)^hz^h,
\end{equation}
where $P_0^{(0)},N_0^{(0)}$ and $S_0^{(0)}$ are the eigenprojection, the nilpotent matrix and the reduced resolvent coefficient associated with the eigenvalue $0$ of $B$ respectively. On the other hand, the formula for $S_0^{(0)}$ is introduced in \eqref{eq:reduced resolvent coefficient}. The expansions \eqref{eq:expansion of the resolvent of T} and \eqref{eq:expansion of the resolvent of B} can be obtained easily due to the properties of the resolvent (see \citep{kato}). Therefore, since the total projection $P_0$ deduced from Proposition \ref{prop:construction of subprojections} can be seen as the Cauchy integral
\begin{equation*}
	P_0(\zeta)=-\dfrac{1}{2\pi i}\int_{\Gamma_0}(T(\zeta)-zI)^{-1}\,dz,
\end{equation*}
where $\Gamma_0$ is an oriented closed curve enclosing $0$ except for the other eigenvalues of $B$ in the resolvent set $\rho(B)$. Hence, since $\Gamma_0$ is a compact set of $\rho(B)$, one can apply \eqref{eq:expansion of the resolvent of T} and \eqref{eq:expansion of the resolvent of B} into the integral formula of $P_0$ and we thus obtain \eqref{eq:expansion of P0} by computing the residue.

{\bf Reduction step:} From Proposition \ref{prop:subprojections} and Proposition \ref{prop:construction of subprojections}, one also has
\begin{equation*}
	\mathbb C^n=\textrm{\normalfont ran}(P_0)\oplus \left(\mathbb C^n-\textrm{\normalfont ran}(P_0)\right),\qquad T=TP_0+T(I-P_0).
\end{equation*}
Thus, the study of the $0$-group of $T$ considered in $\mathbb C^n$ is reduced to the study of the eigenvalues of $TP_0$ considered in $\textrm{\normalfont ran}(P_0)$.

{\bf Step 1:} Under the condition {\bf B}, the eigenvalue $0$ of $B$ is semi-simple {\it i.e.} $BP_0^{(0)}=O_{n\times 1}$ and $\textrm{\normalfont ran}(P_0^{(0)})=\ker(B)$. Thus, based on the expansion \eqref{eq:expansion of P0} of $P_0$ and the fact that $TP_0=P_0TP_0$, one has
\begin{equation*}
	\begin{aligned}
	T(\zeta)P_0(\zeta)&=\bigl(P_0^{(0)}+\zeta P_0^{(1)}+\mathcal O(|\zeta|^2)\bigr)(B+\zeta A)\bigl(P_0^{(0)}+\zeta P_0^{(1)}+\mathcal O(|\zeta|^2)\bigr)\\
	&=\zeta \bigl( C-\zeta D+\mathcal O(|\zeta|^2)\bigr),\qquad |\zeta|\to 0,
	\end{aligned}
\end{equation*}
where $C$ is in \eqref{eq:reduced system} and $D$ is in \eqref{eq:matrix D}.
It follows that $\lambda\in \sigma(TP_0)$ considered in $\textrm{\normalfont ran}(P_0)$ if and only if $\tilde \lambda :=\zeta^{-1}\lambda$ is an eigenvalue of $T_0(\zeta):=C-\zeta D+\mathcal O(|\zeta|^2)$ considered in $\textrm{\normalfont ran}(P_0)$. Therefore, it returns to the eigenvalue problem of $T_0$ considered in the domain $\textrm{\normalfont ran}(P_0)$ and one can apply again Proposition \ref{prop:construction of subprojections}.

Let $c_j$ be the $j$-th element of $\sigma(C)$ considered in $\ker(B)=\textrm{\normalfont ran}(P_0^{(0)})$ for $j\in\{1,\dots,h\}$, then by Proposition \ref{prop:construction of subprojections}, $\tilde \lambda\in\sigma(T_0)$ considered in $\textrm{\normalfont ran}(P_0)$ if and only if $\tilde \lambda \to c_j$ as $|\zeta|\to 0$ for some $j\in\{1,\dots,h\}$. Thus, $\lambda\in\sigma(TP_0)$ considered in $\textrm{\normalfont ran}(P_0)$ if and only if $\zeta^{-1}\lambda\to c_j$ as $|\zeta|\to 0$ for some $j\in\{1,\dots,h\}$. One concludes that the eigenvalues of $TP_0$ considered in $\textrm{\normalfont ran}(P_0)$ are characterized by $c_j$ for $j=1,\dots,h$ and thus they are divided into $h$ groups such that the approximation of the elements of the $j$-th group with respect to $c_j$ has the form
\begin{equation*}
	\lambda_j(\zeta)=c_j\zeta+{\scriptstyle\mathcal O}(|\zeta|),\qquad |\zeta|\to 0.
\end{equation*}
and on the other hand, by Proposition \ref{prop:construction of subprojections}, the total projection associated with this group is approximated by
\begin{equation}\label{eq:expansion of Pj}
	P_j(\zeta)=P_j^{(0)}+\mathcal O(|\zeta|),\qquad |\zeta|\to 0,
\end{equation}
where $P_j^{(0)}$ is the eigenprojection associated with $c_j$ considered in $\ker(B)$ for $j=1,\dots,h$.

{\bf Reduction step:} By Proposition \ref{prop:construction of subprojections}, $T_0$ commutes with $P_j$ for all $j=1,\dots,h$ and one has
\begin{equation*}
	\textrm{\normalfont ran}(P_0)=\bigoplus_{j=1}^h\textrm{\normalfont ran}(P_j),\qquad T_0=\sum_{j=1}^h(T_0P_j).
\end{equation*}
The study of the eigenvalues of $T_0$ considered in $\textrm{\normalfont ran}(P_0)$ is then reduced to the study of the eigenvalues of $T_0P_j$ considered in $\textrm{\normalfont ran}(P_j)$ for $j=1,\dots,h$.

{\bf Final step:} Under the condition {\bf C}, for $j\in\{1,\dots,h\}$, the eigenvalue $c_j$ of $C$ is semi-simple {\it i.e.} $(C-c_jI)P_j^{(0)}=O_{n\times 1}$ and $\textrm{\normalfont ran}(P_j^{(0)})=\ker(C-c_jI)$. Thus, based on the expansion \eqref{eq:expansion of Pj} of $P_j$ and the fact that $T_0P_j=P_jT_0P_j$, one has
\begin{equation*}
	\begin{aligned}
	(T_0(\zeta)-c_jI)P_j(\zeta)&=\bigl(P_j^{(0)}+\mathcal O(|\zeta|)\bigr)\bigl(C-c_jI-\zeta D+\mathcal O(|\zeta|^2)\bigr)\bigl(P_j^{(0)}+\mathcal O(|\zeta|)\bigr)\\
	&=\zeta \bigl( -D_j+\mathcal O(|\zeta|)\bigr),\qquad |\zeta|\to 0.
	\end{aligned}
\end{equation*}
where $D_j:=P_j^{(0)}DP_j^{(0)}$. It follows that $\lambda\in \sigma(T_0P_j)$ considered in $\textrm{\normalfont ran}(P_j)$ if and only if $\tilde \lambda :=\zeta^{-1}(\lambda-c_j)$ is an eigenvalue of $T_j(\zeta):=-D_j+\mathcal O(|\zeta|)$ considered in $\textrm{\normalfont ran}(P_j)$. Therefore, it returns to the eigenvalue problem of $T_j$ considered in the domain $\textrm{\normalfont ran}(P_j)$ and one can apply again Proposition \ref{prop:construction of subprojections}.

For $j\in\{1,\dots,h\}$, let $h_j$ be the cardinality of the spectrum of $D_j$ considered in $\ker(C-c_jI)=\textrm{\normalfont ran}(P_j^{(0)})$ and let $d_{j\ell}$ be the $\ell$-th element of the spectrum for $\ell=1,\dots,h_j$. Then by Proposition \ref{prop:construction of subprojections}, $\tilde \lambda\in\sigma(T_j)$ considered in $\textrm{\normalfont ran}(P_j)$ if and only if $\tilde \lambda \to -d_{j\ell}$ as $|\zeta|\to 0$ for some $\ell\in\{1,\dots,h_j\}$. Thus, $\lambda\in\sigma(T_0P_j)$ considered in $\textrm{\normalfont ran}(P_j)$ if and only if $\zeta^{-1}(\lambda-c_j)\to -d_{j\ell}$ as $|\zeta|\to 0$ for some $\ell\in\{1,\dots,h_j\}$. One concludes that the eigenvalues of $T_0P_j$ considered in $\textrm{\normalfont ran}(P_j)$ are characterized by $d_{j\ell}$ for $\ell=1,\dots,h_j$ and thus they are divided into $h_j$ groups such that the approximation of the elements of the $\ell$-th group with respect to $d_{j\ell}$ has the form
\begin{equation*}
	\lambda_{j\ell}(\zeta)=c_j\zeta-d_{j\ell}\zeta^2+{\scriptstyle\mathcal O}(|\zeta|^2),\qquad |\zeta|\to 0.
\end{equation*}
and on the other hand, by Proposition \ref{prop:construction of subprojections}, the total projection associated with this group is approximated by
\begin{equation}
	P_{j\ell}(\zeta)=P_{j\ell}^{(0)}+\mathcal O(|\zeta|),\qquad |\zeta|\to 0,
\end{equation}
where $P_{j\ell}^{(0)}$ is the eigenprojection associated with $d_{j\ell}$ considered in $\ker(C-c_jI)$ for $\ell=1,\dots,h_j$.

We then deduce from the above steps of approximation for $E(i\xi)=-T(i\xi)$ by multiplying $\lambda_{j\ell}(i\xi)$ by $-1$ to obtain \eqref{eq:expansion of lambdajell}, and \eqref{eq:expansion of Pjell} is the same as $P_{j\ell}(i\xi)$ for each $j\in\{1,\dots,h\}$ and $\ell=1,\dots,h_j$.

Finally, we prove the estimate \eqref{eq:real part of djell}. For $j\in\{1,\dots,h\}$ and $\ell\in\{1,\dots,h_j\}$, since $\lambda_{j\ell}$ in \eqref{eq:expansion of lambdajell} can be seen as an eigenvalue of $E$ and since $c_j$ is real by the condition {\bf C}, if the condition {\bf D} holds, then for $|\xi|$ small, one has
\begin{equation*}
	\textrm{\normalfont Re}\,(\lambda_{j\ell}(i\xi))=-\textrm{\normalfont Re}(d_{j\ell})|\xi|^2+\textrm{\normalfont Re}\,({\scriptstyle\mathcal O}(|\xi|^2))\le -\dfrac{\theta|\xi|^2}{1+|\xi|^2}.
\end{equation*}
Passing through the limit as $|\xi|\to 0$, one has the desired estimate. The proof is done.
\end{proof}

\begin{remark}\label{rem:low frequency 1}
As a consequence, for $|\xi|$ small, in $\textrm{\normalfont ran}(P_{j\ell})$, the operator $E$ has the representation
\begin{equation}
	E_{j\ell}(i\xi)=(-ic_j \xi-d_{j\ell}\xi^2)I-\xi^2N_{j\ell}^{(0)}+\mathcal O(|\xi|^3),
\end{equation}
where $N_{j\ell}^{(0)}$ is the nilpotent matrix associated with the eigenvalue $d_{j\ell}$ of $P_j^{(0)}DP_j^{(0)}$
considered in $\ker(C-c_jI)$ for $j\in\{1,\dots,h\}$ and $\ell\in\{1,\dots,h_j\}$.
\end{remark}

\begin{proposition}[Low frequency 2]\label{prop:low frequency 2}
Let $k\in\mathbb Z^+$ be the number of the nonzero distinct eigenvalues of $B$. If the condition {\bf B} holds, then there are $k$ groups of the eigenvalues of $E$ such that the approximation of the elements of the $j$-th group has the form
\begin{equation}\label{eq:expansion of etaj}
	\eta_j(i\xi)=-e _j+{\scriptstyle \mathcal O}(1),\qquad |\xi|\to 0,
\end{equation}
where $e_j\in\sigma(B)$ with $\textrm{\normalfont Re}\,(e_j)>0$ for all $j=1,\dots,k$.

Moreover, the total projection associated with the $j$-th group is then approximated by
\begin{equation}\label{eq:expansion of Fj}
	F_j(i\xi)=F_j^{(0)}+\mathcal O\bigl(|\xi|\bigr),\qquad |\xi|\to 0.
\end{equation}
\end{proposition}

\begin{proof}
Similarly to the proof of Proposition \ref{prop:low frequency 1}, we consider the operator $T(\zeta)=B+\zeta A$ where $\zeta=i\xi$. However, in this case, we study the eigenvalues of $T$ such that they converge to the nonzero eigenvalues of $B$ as $|\zeta|\to 0$. Let $e_j$ be the $j$-th element of the spectrum of $B$ except for $0$ for $j=1,\dots,k$. Then by Proposition \ref{prop:construction of subprojections}, for any $\eta\in\sigma(T)$ does not converge to $0$, $\eta\to e_j$ for some $j\in\{1,\dots,k\}$. Hence, the approximation of these eigenvalues of $T$ is
\begin{equation*}
	\eta_j(\zeta)=e_j+{\scriptstyle\mathcal O}(1),\qquad |\zeta|\to 0,
\end{equation*}
and also from Proposition \ref{prop:construction of subprojections}, the total projection associated with this group is approximated by
\begin{equation}
	F_j(\zeta)=F_j^{(0)}+\mathcal O(|\zeta|),\qquad |\zeta|\to 0,
\end{equation}
where $F_j^{(0)}$ is the eigenprojection associated with $e_j$ for $j=1,\dots,k$. In particular, $\textrm{\normalfont Re}\,(e_j)>0$ due to the condition {\bf B}.

Finally, since $E(i\xi)=-T(i\xi)$, we obtain \eqref{eq:expansion of etaj} by multiplying $\eta_j(i\xi)$ by $-1$ and \eqref{eq:expansion of Fj} is the same as $F_j(i\xi)$ for all $j=1,\dots,k$.
\end{proof}

\begin{remark}\label{rem:low frequency 2}
As a consequence, for $|\xi|$ small, in $\textrm{\normalfont ran}(F_j)$, the operator $E$ has the representation
\begin{equation}
	E_j(i\xi)=-e_jI-M_j^{(0)}+\mathcal O(|\xi|),
\end{equation}
where $M_j^{(0)}$ is the nilpotent matrix associated with the eigenvalue $e_j$ of $B$ for $j\in\{1,\dots,k\}$.
\end{remark}

The intermediate-frequency case is obtained as follows.

\begin{proposition}[Intermediate frequency]\label{prop:intermediate frequency}
In the compact domain $\varepsilon\le |\xi|\le R$, there is only a finite number of the exceptional points at which the eigenprojections and the nilpotent parts associated with the eigenvalues of $E$ may have poles even the eigenvalues are continuous there.

On the other hand, in every simple domain excluded the exceptional points, the operator $E$ has $r$ (independent from $\xi$) distinct holomorphic eigenvalues denoted by $\nu_j$ with constant algebraic multiplicity together with holomorphic eigenprojections and nilpotent parts denoted by $\Psi_j$ and $ \Xi_j$ associated with them respectively for $j\in\{1,\dots, r\}$.

If the condition {\bf D} holds, then $\textrm{\normalfont Re}\,(\nu)<0$ for any $\nu\in\sigma(E)$ in the domain $\varepsilon\le |\xi|\le R$.
\end{proposition}
\begin{proof}
See \citep{kato}.
\end{proof}

For the high frequency, in order to analyze the eigenvalues of $E(i\xi)=-(B+i\xi A)$, one can analyze the eigenvalues of the operator $\bar E(i\xi):=Q^{-1}E(i\xi)Q=(-i\xi)(\bar A+(i\xi)^{-1} \bar B)$ where $\bar A$ and $\bar B$ are already introduced in \eqref{eq:A and B bar}.

\begin{proposition}[High frequency]\label{prop:high frequency}
Let $s\in \mathbb Z^+$ be the cardinality of the spectrum of the matrix $\bar A$. If the condition {\bf A} holds, then, for $j\in\{1,\dots,s\}$, there is $s_j\in \mathbb Z^+$ to be less than or equal to the algebraic multiplicity of the $j$-th eigenvalue of $\bar A$ such that there are $s_j$ groups of the eigenvalues of $\bar E$ and the approximation of the elements of the $\ell$-th group has the form
\begin{equation}\label{eq:expansion of mujell}
	\mu_{j\ell}(i\xi)=-i\alpha_j\xi-\beta_{j\ell}+{\scriptstyle\mathcal O}(|\xi|^{-1}),\qquad |\xi|\to +\infty,
\end{equation}
where $\alpha_j\in\sigma(\bar A)$ considered in $\mathbb C^n$ and $\beta_{j\ell}\in\sigma\bigl(\Pi_j^{(0)}\bar B\Pi_j^{(0)}\bigr)$ considered in $\ker(\bar A-\alpha_jI)$ for $\ell=1,\dots,s_j$ with $\Pi_j^{(0)}$ defined in \eqref{eq:the eigenprojection associated with alphaj} is the eigenprojection associated with $\alpha_j$. In particular, if the condition {\bf D} holds, then
\begin{equation}\label{eq:real part of betajell}
	\textrm{\normalfont Re}\,(\beta_{j\ell})\ge \theta>0, \qquad \textrm{ for } \ell=1,\dots,s_j \textrm{ and } j=1,\dots, s.
\end{equation}

Moreover, the total projection associated with the $\ell$-th group is then approximated by
\begin{equation}\label{eq:expansion of Pijell}
	\Pi_{j\ell}(i\xi)=\Pi_{j\ell}^{(0)}+\mathcal O(|\xi|^{-1}),\qquad |\xi|\to +\infty,
\end{equation}
where $\Pi_{j\ell}^{(0)}$ is the eigenprojection associated with $\beta_{j\ell}$ considered in $\ker(\bar A-\alpha_jI)$ for $\ell\in\{1,\dots,s_j\}$ and $j\in\{1,\dots,s\}$.
\end{proposition}
\begin{proof}
Similarly to before, we can consider $T(\zeta):=\bar A+\zeta \bar B$ where $\zeta=(i\xi)^{-1}$ firstly in order to apply Proposition \ref{prop:subprojections} and Proposition \ref{prop:construction of subprojections} since $|\zeta|\to 0$ as $|\xi|\to +\infty$. The proof then consists of two steps of approximation and one reduction step between them.

{\bf First step:} The eigenvalues of $T$ are divided into several groups characterized by $\alpha_j\in\sigma(\bar A)$ for $j=1,\dots,s$. Moreover, for $j\in\{1,\dots,s\}$, the approximation for the elements of the $\alpha_j$-group is
\begin{equation*}
	\mu_j(\zeta)=\alpha_j+{\scriptstyle\mathcal O}(1),\qquad |\zeta|\to 0,
\end{equation*}
and the total projection associated with this group is then approximated by
\begin{equation}\label{eq:expansion of Pij}
	\Pi_j(\zeta)=\Pi_j^{(0)}+\mathcal O(|\zeta|),\qquad |\zeta|\to 0,
\end{equation}
where $\Pi_j^{(0)}$ is the eigenprojection associated with the eigenvalue $\alpha_j$ of $\bar A$. In particular, $\Pi_j^{(0)}$ is exactly the same as \eqref{eq:the eigenprojection associated with alphaj} since the eigenprojection $\Pi_j^{(0)}$ can be computed explicitly by the Cauchy integral
\begin{equation*}
	\Pi_j^{(0)}=-\dfrac{1}{2\pi i}\int_{\Gamma_j}(\bar A-zI)^{-1}\,dz=-\dfrac{1}{2\pi i}\int_{\Gamma_j}\textrm{\normalfont diag}\,(a_1-z,\dots,a_n-z)^{-1}\,dz,
\end{equation*}
where $\Gamma_j$ is an oriented closed curve enclosing $\alpha_j$ except for the other eigenvalues of $\bar A$ in the resolvent set $\rho(\bar A)$. Hence, one obtains that
\begin{equation*}
	(\Pi_j^{(0)})_{hk}=\begin{cases} 1&\textrm{if } h=k,a_h=\alpha_j,\\
	0&\textrm{otherwise},
	\end{cases}=\begin{cases} 1 &\textrm{if } h=k\in\mathcal S_j,\\
	0&\textrm{otherwise}, \end{cases}
\end{equation*}
for all $h,k=1,\dots,n$ due to the definition of $\mathcal S_j$ the $j$-th element of the partition $\{\mathcal S_j:j=1,\dots,s\}$ of $\{1,\dots,n\}$.

{\bf Reduction step:} By Proposition \ref{prop:construction of subprojections}, $T$ commutes with $\Pi_j$ for all $j=1,\dots,s$ and one has
\begin{equation*}
	\mathbb C^n=\bigoplus_{j=1}^s\textrm{\normalfont ran}(\Pi_j),\qquad T=\sum_{j=1}^s(T\Pi_j).
\end{equation*}
It implies that the study of the eigenvalues of $T$ considered in $\mathbb C^n$ is reduced to the study of the eigenvalues of $T\Pi_j$ considered in $\textrm{\normalfont ran}(\Pi_j)$ for $j=1,\dots,s$.

{\bf Final step:} Under the condition {\bf A}, the eigenvalues of $\bar A$ are semi-simple {\it i.e.} $\bar A\Pi_j^{(0)}=\alpha_j \Pi_j^{(0)}$ and $\textrm{\normalfont ran}(\Pi_j^{(0)})=\ker(\bar A-\alpha_jI)$ for all $j=1,\dots,s$. Thus, based on the expansion \eqref{eq:expansion of Pij} of $\Pi_j$ and the fact that $T\Pi_j=\Pi_jT\Pi_j$, one has
\begin{equation*}
	\begin{aligned}
	(T(\zeta)-\alpha_jI)\Pi_j(\zeta)&=\bigl(\Pi_j^{(0)}+\mathcal O(|\zeta|)\bigr)(\bar A-\alpha_jI+\zeta \bar B)\bigl(\Pi_j^{(0)}+\mathcal O(|\zeta|)\bigr)\\
	&=\zeta \bigl(  \Pi_j^{(0)}\bar B\Pi_j^{(0)}+\mathcal O(|\zeta|)\bigr),\qquad |\zeta|\to 0,
	\end{aligned}
\end{equation*}
for $j=1,\dots,s$. It follows that $\mu\in \sigma(T\Pi_j)$ considered in $\textrm{\normalfont ran}(\Pi_j)$ if and only if $\tilde \mu :=\zeta^{-1}(\mu-\alpha_j)$ is an eigenvalue of $T_j(\zeta):=\Pi_j^{(0)}\bar B\Pi_j^{(0)}+\mathcal O(|\zeta|)$ considered in $\textrm{\normalfont ran}(\Pi_j)$ for $j=1,\dots,s$. Therefore, it returns to the eigenvalue problem of $T_j$ considered in the domain $\textrm{\normalfont ran}(\Pi_j)$ for $j=1,\dots,s$ and one can apply again Proposition \ref{prop:construction of subprojections}.

For $j\in\{1,\dots,s\}$, let $s_j$ be the cardinality of the spectrum of $\Pi_j^{(0)}\bar B\Pi_j^{(0)}$ considered in $\ker(\bar A-\alpha_jI)=\textrm{\normalfont ran}(\Pi_j^{(0)})$ and let $\beta_{j\ell}$ be the $\ell$-th elements of the spectrum for $\ell=1,\dots,s_j$. Then, by Proposition \ref{prop:construction of subprojections}, $\tilde\mu\in\sigma(T_j)$ considered in $\textrm{\normalfont ran}(\Pi_j)$ if and only if $\tilde\mu\to \beta_{j\ell}$ as $|\zeta|\to 0$ for some $\ell\in\{1,\dots,s_j\}$. Thus, $\mu\in \sigma(T\Pi_j)$ considered in $\textrm{\normalfont ran}(\Pi_j)$ if and only if $\zeta^{-1}(\mu-\alpha_j)\to \beta_{j\ell}$ as $|\zeta|\to 0$ for some $\ell\in\{1,\dots,s_j\}$. It implies the eigenvalues of $T\Pi_j$ considered in $\textrm{\normalfont ran}(\Pi_j)$ are characterized by $\beta_{j\ell}$ such that the approximation of the elements of the $\ell$-th group with respect to $\beta_{j\ell}$ is
\begin{equation*}
	\mu_{j\ell}(\zeta)=\alpha_j+\beta_{j\ell}\zeta+{\scriptstyle\mathcal O}(|\zeta|),\qquad |\zeta|\to 0,
\end{equation*}
and also by Proposition \ref{prop:construction of subprojections} that the total projection associated with this group is approximated by
\begin{equation*}
	\Pi_{j\ell}(\zeta)=\Pi_{j\ell}^{(0)}+\mathcal O(|\zeta|),\qquad |\zeta|\to 0,
\end{equation*}
where $\Pi_{j\ell}^{(0)}$ is the eigenprojection associated with $\beta_{j\ell}$ considered in $\ker(\bar A-\alpha_jI)$ for $\ell=1,\dots,s_j$.

We deduce from the above steps of approximation for $\bar E(i\xi)=(-i\xi)T\bigl((i\xi)^{-1}\bigr)$ by multiplying $\lambda_{j\ell}\bigl((i\xi)^{-1}\bigr)$ by $(-i\xi)$ to obtain \eqref{eq:expansion of mujell}, and \eqref{eq:expansion of Pijell} is the same as $\Pi_{j\ell}\bigl((i\xi)^{-1})$ for each $j\in\{1,\dots,s\}$ and $\ell=1,\dots,s_j$.

Finally, we prove the estimate \eqref{eq:real part of betajell}. For $j\in\{1,\dots,s\}$ and $\ell\in\{1,\dots,s_j\}$, since $\mu_{j\ell}$ in \eqref{eq:expansion of mujell} can be seen as an eigenvalue of $\bar E$ and thus of $E=Q\bar EQ^{-1}$ and since $\alpha_j$ is real by the condition {\bf A}, if the condition {\bf D} holds, then for $|\xi|$ large, one has
\begin{equation*}
	\textrm{\normalfont Re}\,(\mu_{j\ell}(i\xi))=-\textrm{\normalfont Re}(\beta_{j\ell})+\textrm{\normalfont Re}\,({\scriptstyle\mathcal O}(1))\le -\dfrac{\theta|\xi|^2}{1+|\xi|^2}.
\end{equation*}
Passing through the limit as $|\xi|\to +\infty$, one has the desired estimate. We finish the proof.
\end{proof}

\begin{remark}\label{rem:high frequency}
As a consequence, for $|\xi|$ large, in $\textrm{\normalfont ran}(\Pi_{j\ell})$, the operator $E$ has the representation
\begin{equation}
	E_{j\ell}(i\xi)=(-i\alpha_j \xi-\beta_{j\ell})I-\Theta_{j\ell}^{(0)}+\mathcal O(|\xi|^{-1}),
\end{equation}
where $\Theta_{j\ell}^{(0)}$ is the nilpotent matrix associated with the eigenvalue $\beta_{j\ell}$ of $\Pi_j^{(0)}\bar B\Pi_j^{(0)}$ considered in $\ker(\bar A-\alpha_jI)$ for $j\in\{1,\dots,s\}$ and $\ell\in\{1,\dots,s_j\}$.
\end{remark}

\section{Fundamental solution}\label{sec:fundamental solution}
The aim of this section is to introduce the estimates for the fundamental solution to \eqref{eq:original hyperbolic system} in the frequency space. Let consider the fundamental system
\begin{equation}\label{eq:fundamental system}
	\partial_t\hat G-E\hat G=0,\qquad \hat G_{t=0}=I,
\end{equation}
where $E=E(i\xi)=-(B+i\xi A)$ with $\xi\in\mathbb R$.

One sets the following kernel
\begin{equation}\label{eq:parabolic kernel}
	\hat K(\xi,t):=\sum_{j,\ell=1}^{h,h_j}e^{(-ic_j\xi-d_{j\ell}\xi^2)t}e^{-N_{j\ell}^{(0)}\xi^2t}P_{j\ell}^{(0)},
\end{equation}
and the kernel
\begin{equation}\label{eq:exponential decay kernel}
	\hat V(\xi,t):=Q\sum_{j,\ell=1}^{s,s_j}e^{(-i\alpha_j\xi-\beta_{j\ell})t}e^{-\Theta_{j\ell}^{(0)}t}\Pi_{j\ell}^{(0)}Q^{-1},
\end{equation}
where the coefficients are introduced in the previous section.

Moreover, we introduce the two useful lemmas used in this section as follows.
\begin{lemma}\label{lem:nilpotent}
	If $X$ is a constant complex nilpotent matrix, then for all $\varepsilon'>0$, there exists $C=C(\varepsilon')>0$ such that
\begin{equation*}
	\bigl| e^{c X+Y}-e^{cX}\bigr|\le Ce^{\varepsilon' |c|+C|Y|}|Y|
\end{equation*}
and
\begin{equation*}
	\bigl| e^{c X+Y}-e^{cX}-e^{cX}Y\bigr|\le Ce^{\varepsilon' |c|+C|Y|}|Y|^2
\end{equation*}
for every complex constant $c:=c(t)$ and matrix $Y:=Y(t)$ for $t>0$.
\end{lemma}
\begin{proof}
The proof is based on the existence of a basis of $\mathbb C^n$ such that $|X|\le \varepsilon'$ for any fixed $\varepsilon'>0$ once written in this basis, then the constant $C(\varepsilon')$ can be chosen as the product of the norm of the changing basis matrix and the norm of its inverse for any matrix norm. The second inequality due to the fact that the first order derivative $\textrm{\normalfont d}_{\exp}$ at $X$ of the application $X\to e^X$ is $e^X$ and thus one has
\begin{equation*}
	\bigl| e^{X+Y}-e^X-e^XY\bigr|\le C|Y|^2\sup_{s\in[0,1]}\left|\textrm{\normalfont d}_{\exp}^2(X+sY)\right|\le C|Y|^2e^{|X|+|Y|}
\end{equation*}
where $\textrm{\normalfont d}_{\exp}^2$ is the second order derivative of $X\to e^{X}$. Thus, under a change of basis, one obtains the desired estimiate. One can find a detailed proof in \citep{bianchini07}.
\end{proof}

\begin{lemma}\label{lem:main estimates}
For $0<\varepsilon<R<+\infty$, if  $a,b\ge 0$ and $c,d>0$ and $r\in[1,\infty]$, there exists $C:=C(r)>0$ such that for all $t\ge 1$, one has
	\begin{equation}
		\bigl\| |\cdot|^at^be^{-c|\cdot|^{d}t}\bigr\|_{L^{r}}\le Ct^{-\frac{1}{d}\frac{1}{r}+b-\frac{a}{d}}.
	\end{equation}
\end{lemma}
\begin{proof}
By changing of variables.
\end{proof}

\begin{proposition}[Fundamental solution estimates]\label{prop:fundamental solution standard type}
For $0<\varepsilon<R<+\infty$, for $r\in[1,\infty]$ and $t\ge 1$, there exists positive constants $C:=C(r)$ and $\delta$ such that if the conditions {\bf A}, {\bf B}, {\bf C} and {\bf D} are satisfied, the following hold.

\noindent
1. For $|\xi|<\varepsilon$, one has
\begin{equation}\label{est:low}
	\bigl\| \hat G-\hat K\bigr\|_{L^{r}} \le Ct^{-\frac12\frac{1}{r}-\frac12},\qquad \bigl\|\hat V\bigr\|_{L^{r}}\le Ce^{-\delta t}.
\end{equation}

\noindent
2. For $\varepsilon\le |\xi|\le R$, one has
\begin{equation}\label{est:intermediate}
	\bigl\|\hat G\bigr\|_{L^{r}},\bigl\|\hat K\bigr\|_{L^{r}},\bigl\|\hat V\bigr\|_{L^{r}}\le Ce^{-\delta t}.
\end{equation}

\noindent
3. For $|\xi|>R$, one has
\begin{equation}\label{est:high 1}
	\bigl\|\hat G-\hat V\bigr\|_{L^{r}}\le Ce^{-\delta t} \quad \textrm{for} \quad r>1, \qquad \bigl\|\hat K\bigr\|_{L^{r}}\le Ce^{-\delta t}.\end{equation}
Moreover, we also have
\begin{equation}\label{est:high 2}
	\bigl\| \mathcal F^{-1}(\hat G-\hat V)\bigr\|_{L^\infty}\le Ce^{-\delta t}.
\end{equation}
\end{proposition}
\begin{proof}
For $|\xi|<\varepsilon$, by Remark \ref{rem:low frequency 1} and Remark \ref{rem:low frequency 2}, the solution $\hat G$ to the system \eqref{eq:fundamental system} is given by $\hat G=\hat G_1+\hat G_2$ where
\begin{equation}\label{eq:G1 small}
	\hat G_1(\xi,t):=\sum_{j,\ell=1}^{h,h_j}e^{(-ic_j \xi-d_{j\ell}\xi^2)t}e^{-N_{j\ell}^{(0)}\xi^2t+\mathcal O(|\xi|^3)t}\bigl(P_{j\ell}^{(0)}+\mathcal O(|\xi|)\bigr),
\end{equation}
and
\begin{equation}\label{eq:G2 small}
	\hat G_2(\xi,t):=\sum_{j=1}^ke^{-e_jt}e^{-M_j^{(0)}t+\mathcal O(|\xi|)t}\bigl(F_j^{(0)}+\mathcal O(|\xi|)\bigr).
\end{equation}
It follows that $\hat G-\hat K=\hat G_1-\hat K+\hat G_2=I_1+I_2+J$ where $J=\hat G_2$ and
\begin{equation}\label{eq:I_1 small}
	I_1:=\sum_{j,\ell=1}^{h,h_j}e^{(-ic_j \xi-d_{j\ell}\xi^2)t}\left(e^{-N_{j\ell}^{(0)}\xi^2t+\mathcal O(|\xi|^3)t}-e^{-N_{j\ell}^{(0)}\xi^2t}\right)P_{j\ell}^{(0)}
\end{equation}
and
\begin{equation}\label{eq:I_2 small}
	I_2:=\sum_{j,\ell=1}^{h,h_j}e^{(-ic_j \xi-d_{j\ell}\xi^2)t}e^{-N_{j\ell}^{(0)}\xi^2t+\mathcal O(|\xi|^3)t}\mathcal O(|\xi|).
\end{equation}

Firstly, we estimate for $I_1$ with $|\xi|<\varepsilon$ small enough by taking the matrix norm both sides of \eqref{eq:I_1 small}. Since $c_j\in\mathbb R$ for all $j\in\{1,\dots,h\}$, one has
\begin{equation*}
	\bigl| I_1\bigr|\le C\sum_{j,\ell=1}^{h,h_j}e^{-\textrm{\normalfont Re}\,(d_{j\ell})|\xi|^2t}\left|e^{-N_{j\ell}^{(0)}\xi^2t+\mathcal O(|\xi|^3)t}-e^{-N_{j\ell}^{(0)}\xi^2t}\right|.
\end{equation*}
On the other hand, from Proposition \ref{prop:low frequency 1}, $\textrm{\normalfont Re}\,(d_{j\ell})\ge \theta>0$ and $N_{j\ell}^{(0)}$ is a nilpotent matrix for all $j\in\{1,\dots,h\}$ and $\ell\in\{1,\dots,h_j\}$. Thus, by choosing $\varepsilon'=\frac14\textrm{\normalfont Re}\,(d_{j\ell})$ for each $j\in\{1,\dots,h\}$ and $\ell\in\{1,\dots,h_j\}$, from Lemma \ref{lem:nilpotent}, we have
\begin{equation*}
\begin{aligned}
	\bigl| I_1\bigr|&\le C\sum_{j,\ell=1}^{h,h_j}e^{-\textrm{\normalfont Re}\,(d_{j\ell})|\xi|^2t}e^{\frac14\textrm{\normalfont Re}\,(d_{j\ell})|\xi|^2t+C|\xi|^3t}|\xi|^3t\\
	&\le C\sum_{j,\ell=1}^{h,h_j}e^{-\textrm{\normalfont Re}\,(d_{j\ell})|\xi|^2t}e^{\frac14\textrm{\normalfont Re}\,(d_{j\ell})|\xi|^2t+C\varepsilon |\xi|^2t}|\xi|^3t\le Ce^{-\frac12\theta|\xi|^2t}|\xi|^3t.
\end{aligned}
\end{equation*}
Hence, by applying Lemma \ref{lem:main estimates}, we have
\begin{equation}\label{est:I_1 small}
	\|I_1\|_{L^{r}}\le Ct^{-\frac12\frac{1}{r}-\frac12},\quad \textrm{for }  r\in [1,\infty] \textrm{ and } t\ge 1.
\end{equation}

Similarly, we also have the estimate for $I_2$ with $|\xi|<\varepsilon$ small enough. Indeed, from \eqref{eq:I_2 small}, one has
\begin{equation*}
\begin{aligned}
	|I_2|&\le C\sum_{j,\ell=1}^{h,h_j}e^{-\textrm{\normalfont Re}\,(d_{j\ell})|\xi|^2t}e^{\bigl|N_{j\ell}^{(0)}\bigr||\xi|^2t+C|\xi|^3t}|\xi|\\
	&\le C\sum_{j,\ell=1}^{h,h_j}e^{-\textrm{\normalfont Re}\,(d_{j\ell})|\xi|^2t}e^{C\varepsilon|\xi|^2t}|\xi|\le Ce^{-\frac12\theta |\xi|^2t}|\xi|
\end{aligned}
\end{equation*}
since one can assume that $\bigl|N_{j\ell}^{(0)}\bigr|$ is small for all $j\in\{1,\dots,h\}$ and $\ell\in\{1,\dots,h_j\}$ based on the fact that they are nilpotent matrices. Hence, by applying Lemma \ref{lem:main estimates}, we have
\begin{equation}\label{est:I_2 small}
	\|I_2\|_{L^{r}}\le Ct^{-\frac12\frac{1}{r}-\frac12},\quad \textrm{for }  r\in [1,\infty] \textrm{ and } t\ge 1.
\end{equation}

We estimate for $J$. From \eqref{eq:G2 small}, we have
\begin{equation*}
	|J|\le C\sum_{j=1}^ke^{-\textrm{\normalfont Re}\,(e_j)t}e^{\bigl| M_j^{(0)}\bigr|t+C|\xi|t}(1+|\xi|).
\end{equation*}
Then, by Proposition \ref{prop:low frequency 2}, $\textrm{\normalfont Re}\,(e_j)>0$ and $M_j^{(0)}$ is a nilpotent matrix for $j\in\{1,\dots,k\}$, we can assume that $\bigl| M_j^{(0)}\bigr|$ is small, and thus, since $|\xi|<\varepsilon$ small enough, we obtain
\begin{equation}\label{est:J small}
	\|J\|_{L^r}\le C\sum_{j=1}^ke^{-\textrm{\normalfont Re}\,(e_j)t}e^{C\varepsilon t}\bigl(1+\varepsilon\bigr)\le Ce^{-\delta t}
\end{equation}
for some $\delta >0$.

Therefore,  from \eqref{est:I_1 small}, \eqref{est:I_2 small} and \eqref{est:J small}, one obtains for $|\xi|<\varepsilon$ that
\begin{equation*}
\bigl\|\hat G-\hat K\bigr\|_{L^{r}}\le \|I_1\|_{L^{r}}+\|I_2\|_{L^{r}}+\|J\|_{L^{r}}\le Ct^{-\frac12\frac{1}{r}-\frac12},\quad \textrm{for }  r\in [1,\infty] \textrm{ and } t\ge 1.
\end{equation*}

We now estimate for $\hat V$ in \eqref{eq:exponential decay kernel} with $|\xi|<\varepsilon$. Since $\alpha_j\in\mathbb R$ for all $j\in\{1,\dots,s\}$, one has
\begin{equation*}
	\bigl|\hat V(\xi,t)\bigr|\le C\sum_{j,\ell=1}^{s,s_j}e^{-\textrm{\normalfont Re}\,(\beta_{j\ell})t}e^{\bigl|\Theta_{j\ell}^{(0)}\bigr|t}.
\end{equation*}
Thus, by Proposition \ref{prop:high frequency}, since $\textrm{\normalfont Re}\,(\beta_{j\ell})\ge \theta>0$ and $\Theta_{j\ell}^{(0)}$ is a nilpotent matrix for all $j\in\{1,\dots,h\}$ and $\ell\in\{1,\dots,h_j\}$, one obtains
\begin{equation*}
	\bigl\|\hat V\bigr\|_{L^{r}}\le Ce^{-\frac12\theta t},\quad \textrm{for }  r\in [1,\infty] \textrm{ and } t\ge 1,
\end{equation*}
since we can assume that $\bigl|\Theta_{j\ell}^{(0)}\bigr|$ is small for all $j\in\{1,\dots,s\}$ and $\ell\in\{1,\dots,s_j\}$ similarly to before.

In the compact domain $\varepsilon \le |\xi|\le R$, there are the exceptional points where the eigenprojections and the nilpotent parts associated with the eigenvalues of $E(i\xi)=B+i\xi A$ in this domain may not be defined even the eigenvalues are continuous there. However, the number of these exceptional points is always finite in $\varepsilon\le |\xi|\le R$ as introduced in Proposition \ref{prop:intermediate frequency}, once integrating, for $\hat G=e^{Et}$ and for some $\delta>0$, from the condition {\bf D}, we still obtain
\begin{equation*}
	\bigl\|\hat G\bigr\|_{L^{r}}\le \left\| e^{-\frac{\theta|\cdot|^2}{1+|\cdot|^2}t}\right\|_{L^{r}}\le Ce^{-\delta t},\quad \textrm{for }  r\in [1,\infty] \textrm{ and } t\ge 1.
\end{equation*}

For $\hat K$ in \eqref{eq:parabolic kernel}, similarly to the small frequency, for $\varepsilon\le |\xi|\le R$, one has
\begin{equation}\label{est:Kintermediate}
	\bigl|\hat K(\xi,t)\bigr|\le C\sum_{j,\ell=1}^{h,h_j}e^{-\textrm{\normalfont Re}\,(d_{j\ell})|\xi|^2t}e^{\bigl|N_{j\ell}^{(0)}\bigr||\xi|^2t}\le Ce^{-\frac12\theta \varepsilon^2t}
\end{equation}
since $\textrm{\normalfont Re}\,(d_{j\ell})\ge \theta$ and $N_{j\ell}^{(0)}$ is a nilpotent matrix for all $j\in\{1,\dots,h\}$ and $\ell\in\{1,\dots,h_j\}$. Thus, for $\varepsilon\le |\xi|\le R$, one obtains
\begin{equation*}
	\bigl\|\hat K\bigr\|_{L^{r}}\le Ce^{-\frac12\theta\varepsilon^2t},\quad \textrm{for }  r\in [1,\infty] \textrm{ and } t\ge 1.
\end{equation*}

For $\hat V$ in \eqref{eq:exponential decay kernel}, similarly to the small frequency, for $\varepsilon\le |\xi|\le R$, one has
\begin{equation}\label{est:Vintermediate}
	\bigl|\hat V(\xi,t)\bigr|\le C\sum_{j,\ell=1}^{s,s_j}e^{-\textrm{\normalfont Re}\,(\beta_{j\ell})t}e^{\bigl|\Theta_{j\ell}^{(0)}\bigr|t}\le Ce^{-\frac12\theta t}
\end{equation}
since $\textrm{\normalfont Re}\,(\beta_{j\ell})\ge \theta$ and $\Theta_{j\ell}^{(0)}$ is a nilpotent matrix for all $j\in\{1,\dots,s\}$ and $\ell\in\{1,\dots,s_j\}$. Hence, for $\varepsilon\le |\xi|\le R$, we have
\begin{equation*}
	\bigl\|\hat V\bigr\|_{L^{r}}\le Ce^{-\frac12\theta t},\quad \textrm{for }  r\in [1,\infty] \textrm{ and } t\ge 1.
\end{equation*}

Finally, we study the case $|\xi|>R$. By Remark \ref{rem:high frequency},  the solution $\hat G$ to the system \eqref{eq:fundamental system} is given by
\begin{equation}\label{eq:G large}
	\hat G(\xi,t):=\sum_{j,\ell=1}^{s,s_j}e^{(-i\alpha_j\xi-\beta_{j\ell})t}e^{-\Theta_{j\ell}^{(0)}t+\mathcal O(|\xi|^{-1})t}\bigl(\Pi_{j\ell}^{(0)}+\mathcal O(|\xi|^{-1})\bigr).
\end{equation}
Then, we have $\hat G-\hat V=I+J$ where
\begin{equation}\label{eq:I large}
	I:=\sum_{j,\ell=1}^{s,s_j}e^{(-i\alpha_j\xi-\beta_{j\ell})t}\left( e^{-\Theta_{j\ell}^{(0)}t+\mathcal O(|\xi|^{-1})t}-e^{-\Theta_{j\ell}^{(0)}t}\right)\Pi_{j\ell}^{(0)}
\end{equation}
and
\begin{equation}\label{eq:J large}
	J:=\sum_{j,\ell=1}^{s,s_j}e^{(-i\alpha_j\xi-\beta_{j\ell})t}e^{-\Theta_{j\ell}^{(0)}t+\mathcal O(|\xi|^{-1})t}\mathcal O(|\xi|^{-1}).
\end{equation}

We estimate for $I$ firstly and then for $J$. Since $\alpha_j\in\mathbb R$ for all $j\in\{1,\dots,s\}$, we have
\begin{equation*}
	|I|\le C\sum_{j,\ell=1}^{s,s_j}e^{-\textrm{\normalfont Re}\,(\beta_{j\ell})t}\left| e^{-\Theta_{j\ell}^{(0)}t+\mathcal O(|\xi|^{-1})t}-e^{-\Theta_{j\ell}^{(0)}t}\right|.
\end{equation*}
On the other hand, from Proposition \ref{prop:high frequency}, $\textrm{\normalfont Re}\,(\beta_{j\ell})\ge \theta>0$ and $\Theta_{j\ell}^{(0)}$ is a nilpotent matrix for all $j\in\{1,\dots,s\}$ and $\ell\in\{1,\dots,s_j\}$. Let $\varepsilon'=\frac14\textrm{\normalfont Re}\,(\beta_{j\ell})$ and applying Lemma \ref{lem:nilpotent}, for $|\xi|>R$ large enough, we obtain
\begin{equation*}
\begin{aligned}
	|I|&\le C\sum_{j,\ell=1}^{s,s_j}e^{-\textrm{\normalfont Re}\,(\beta_{j\ell})t}e^{\frac14\textrm{\normalfont Re}\,(\beta_{j\ell})t+C|\xi|^{-1}t}|\xi|^{-1}t\\
	&\le C\sum_{j,\ell=1}^{s,s_j}e^{-\textrm{\normalfont Re}\,(\beta_{j\ell})t}e^{\frac14\textrm{\normalfont Re}\,(\beta_{j\ell})t+CR^{-1}t}|\xi|^{-1}t\le Ce^{-\frac12\theta t}|\xi|^{-1}t.
\end{aligned}
\end{equation*}
Thus, for $|\xi|>R$ large enough, one has
\begin{equation}\label{est:I large}
	\|I\|_{L^{r}}\le Ce^{-\frac14\theta t}, \quad \textrm{for }  r\in (1,\infty] \textrm{ and } t\ge 1.
\end{equation}

Similarly, we estimate for $J$ for $|\xi|>R$ large enough. From \eqref{eq:J large}, one has
\begin{equation*}
\begin{aligned}
	|J|&\le \sum_{j,\ell=1}^{s,s_j}e^{-\textrm{\normalfont Re}\,(\beta_{j\ell})t}e^{\bigl|\Theta_{j\ell}^{(0)}\bigr|t+C|\xi|^{-1}t}|\xi|^{-1}\\
	&\le \sum_{j,\ell=1}^{s,s_j}e^{-\textrm{\normalfont Re}\,(\beta_{j\ell})t}e^{CR^{-1}t}|\xi|^{-1}\le Ce^{-\frac12\theta t}|\xi|^{-1}
\end{aligned}
\end{equation*}
since one can assume that $\bigl|\Theta_{j\ell}^{(0)}\bigr|$ is small for all $j\in\{1,\dots,s\}$ and $\ell\in\{1,\dots,s_j\}$. Thus, for $|\xi|>R$ large enough, one has
\begin{equation}\label{est:J large}
	\|J\|_{L^{r}}\le Ce^{-\frac12\theta t}, \quad \textrm{for }  r\in (1,\infty] \textrm{ and } t\ge 1.
\end{equation}

Therefore, from \eqref{est:I large} and \eqref{est:J large}, there is a constant $\delta>0$ such that
\begin{equation*}
	\bigl\| \hat G-\hat V\bigr\|_{L^{r}}\le \|I\|_{L^{r}}+\|J\|_{L^{r}}\le Ce^{-\delta t},\quad \textrm{for }  r\in (1,\infty] \textrm{ and } t\ge 1.
\end{equation*}

On the other hand, we estimate for $\hat K$ in \eqref{eq:parabolic kernel} with $|\xi|>R$. We have
\begin{equation}\label{est:K large}
\begin{aligned}
	|\hat K(\xi,t)|&\le C\sum_{j,\ell=1}^{h,h_j}e^{-\textrm{\normalfont Re}\,(d_{j\ell})|\xi|^2t}e^{\bigl|N_{j\ell}^{(0)}\bigr||\xi|^2t}\\
	&\le Ce^{-\frac12\theta R^2t}\sum_{j,\ell=1}^{h,h_j}e^{-\frac12\textrm{\normalfont Re}\,(d_{j\ell})|\xi|^2t}e^{\frac14\textrm{\normalfont Re}\,(d_{j\ell})|\xi|^2t}\le Ce^{-\frac12\theta R^2t}e^{-\frac14\theta|\xi|^2t}
\end{aligned}
\end{equation}
since $\textrm{\normalfont Re}\,(d_{j\ell})\ge \theta>0$ and $N_{j\ell}^{(0)}$ that is a nilpotent matrix can be assumed to have $\bigl|N_{j\ell}^{(0)}\bigr|$ small enough for all $j\in\{1,\dots,h\}$ and $\ell\in\{1,\dots,h_j\}$ by Proposition \ref{prop:high frequency}. Thus, for $|\xi|>R$, we have
\begin{equation*}
	\|\hat K\|_{L^{r}}\le Ce^{-\frac12\theta R^2t}t^{-\frac12}\le Ce^{-\frac14\theta R^2t}, \quad \textrm{for }  r\in [1,\infty] \textrm{ and } t\ge 1. 
\end{equation*}

We now estimate the $L^{\infty}$-norm of the function $\mathcal F^{-1}(\hat G-\hat V)=\mathcal F^{-1}(I)+\mathcal F^{-1}(J)$ in $L^{\infty}$ for $|\xi|>R$ large enough where $I$ and $J$ are in \eqref{eq:I large} and \eqref{eq:J large} respectively. Primarily, from \eqref{eq:I large} and by applying the Taylor expansion to the application $X\to e^X$, we have $I=I_1+I_2+I_3$ where
\begin{equation}\label{eq:I_1 large}
	I_1:=t\sum_{j,\ell=1}^{s,s_j}\dfrac{e^{-i\alpha_j\xi t}}{i\xi} e^{-\beta_{j\ell}t}e^{-\Theta_{j\ell}^{(0)}t}M\Pi_{j\ell}^{(0)},
\end{equation}
where $M$ is the coefficient in $\mathcal O(|\xi|^{-1})$ associated with $(i\xi)^{-1}$, and
\begin{equation}\label{eq:I_2 large}
	I_2:=t\sum_{j,\ell=1}^{s,s_j}e^{(-i\alpha_j\xi -\beta_{j\ell})t}e^{-\Theta_{j\ell}^{(0)}t}\mathcal O(|\xi|^{-2})\Pi_{j\ell}^{(0)},
\end{equation}
and
\begin{equation}\label{eq:I_3 large}
	I_3:=\sum_{j,\ell=1}^{s,s_j}e^{(-i\alpha_j\xi-\beta_{j\ell})t}\left( e^{-\Theta_{j\ell}^{(0)}t+\mathcal O(|\xi|^{-1})t}-e^{-\Theta_{j\ell}^{(0)}t}-e^{-\Theta_{j\ell}^{(0)}t}\mathcal O(|\xi|^{-1})t\right)\Pi_{j\ell}^{(0)}.
\end{equation}

We first estimate for $\mathcal F^{-1}(I_1)=\sum_{j=1}^s\mathcal F_{j}^{-1}(I_1)$ where
\begin{equation*}
	\mathcal F_{j}^{-1}(I_1):=t\sum_{\ell=1}^{s_j}\mathcal F^{-1}\left(\dfrac{e^{-i\alpha_j\xi t}}{i\xi}\right) e^{-\beta_{j\ell}t}e^{-\Theta_{j\ell}^{(0)}t}M\Pi_{j\ell}^{(0)}.
\end{equation*}
 for $j\in\{1,\dots,s\}$ .
 
For each $j\in\{1,\dots,s\}$, one has
\begin{equation}\label{est:I_1 large}
\begin{aligned}
	\bigl|\mathcal F_j^{-1}(I_1)(x,t)\bigr|&\le Ct\sum_{\ell=1}^{s_j}\left|\int_{-\infty}^{-R}+\int_R^{+\infty}\dfrac{e^{i(x-\alpha_jt)\xi}}{i\xi}\,d\xi\right|e^{-\textrm{\normalfont Re}\,(\beta_{j\ell})t}e^{\bigl|\Theta_{j\ell}^{(0)}\bigr|t}\\
	&\le Ct\sum_{\ell=1}^{s_j}\left|2\int_R^{+\infty}\dfrac{\sin((x-\alpha_jt)\xi)}{\xi}\,d\xi\right|e^{-\textrm{\normalfont Re}\,(\beta_{j\ell})t}e^{\bigl|\Theta_{j\ell}^{(0)}\bigr|t}\\
	&\le Cte^{-\frac12\theta t}|x-\alpha_jt|
\end{aligned}
\end{equation}
since $\textrm{\normalfont Re}\,(\beta_{j\ell})\ge \theta>0$ and $\Theta_{j\ell}^{(0)}$ that is a nilpotent matrix with norm can be chosen small enough for all $j\in\{1,\dots,s\}$ and $\ell\in\{1,\dots,s_j\}$ by Proposition \ref{prop:high frequency}. Hence, if $|x|\le Ct$ where $C$ is a positive constant, then, for all $j\in\{1,\dots,s\}$, we have
\begin{equation*}
	\bigl\|\mathcal F_j^{-1}(I_1)\bigr\|_{L^{\infty}}\le Ct^2e^{-\frac12\theta t}\le Ce^{-\frac14\theta t}.
\end{equation*}

We now estimate for $\mathcal F_j^{-1}(I_1)$ in the case where $|x|>Ct$ and $C$ large enough for $j\in\{1,\dots,s\}$. Noting that in this case we have
\begin{equation}\label{est:alphaj=0}
	e^{x\alpha_j}\le e^{|x||\alpha_j|}\le e^{\frac{|x|^2}{t}|\alpha_j||x|^{-1}t}\le e^{\varepsilon\frac{|x|^2}{t}}
\end{equation}
where $\varepsilon$ is small enough.

One has
\begin{equation}\label{eq:I_1^a |x| large}
	\bigl|\mathcal F_j^{-1}(I_1)(x,t)\bigr|\le Ct\sum_{\ell=1}^{s_j}\left|\int_{-\infty}^{-R}+\int_R^{+\infty}\dfrac{e^{i(x-\alpha_jt)\xi}}{i\xi}\,d\xi\right|e^{-\textrm{\normalfont Re}\,(\beta_{j\ell})t}e^{\bigl|\Theta_{j\ell}^{(0)}\bigr|t}.
\end{equation}
We estimate for the integral
\begin{equation}\label{eq:H}
	H:=\int_{-\infty}^{-R}+\int_R^{+\infty}\dfrac{e^{i(x-\alpha_jt)\xi}}{i\xi}\,d\xi=\lim_{K\to +\infty}\int_{-K}^{-R}+\int_R^{K}\dfrac{e^{i(x-\alpha_jt)\xi}}{i\xi}\,d\xi=H_1+H_2.
\end{equation}
Due to the fact that the integrand is holomorphic, we can estimate $H_2$ by considering $\xi=\zeta+i\eta\in\mathbb C$ and by changing the path $\{(\zeta,0):\zeta\textrm{ from }R \textrm{ to } K\}$ to the path $\gamma:=\gamma_1\cup \gamma_2\cup\gamma_3$ in the complex plane where
\begin{equation}
	\gamma_1:=\left\{(\zeta,\eta):\zeta=R,\eta\textrm{ from }0 \textrm{ to } \frac xt\right\},
\end{equation}
\begin{equation}
	\gamma_2:=\left\{(\zeta,\eta):\zeta\textrm{ from } R \textrm{ to } K,\eta=\frac xt\right\}
\end{equation}
and
\begin{equation}
	\gamma_3:=\left\{(\zeta,\eta):\zeta=K,\eta\textrm{ from }\frac xt \textrm{ to } 0\right\}.
\end{equation}

Then, by parameterizing $\gamma_1(s)=R+i\frac{x}{t}s$ for $s\in[0,1]$, since $|x|> Ct$, we have
\begin{equation}\label{est:gamma1}
\begin{aligned}
	\left|\lim_{K\to +\infty} \int_{\gamma_1}\dfrac{e^{i(x-\alpha_jt)\xi}}{i\xi}\,d\xi\right |&=\left| \int_{0}^{1}\dfrac{e^{i(x-\alpha_jt)R+x\alpha_js-\frac{|x|^2}{t}s}}{R+i\frac{x}{t}s}\,\dfrac{x}{t}ds\right|\\
	&\le \dfrac{C}{R}\int_{0}^{1}\left(\dfrac{|x|}{t}+\dfrac{|x|^2}{t^2}\right)e^{\varepsilon\frac{|x|^2}{t}s}e^{-\frac{|x|^2}{t}s}\,ds\\
	&\le \dfrac{C}{R}\left(\dfrac{1}{|x|}+\dfrac{1}{t}\right)\bigl(1-e^{-\frac{|x|^2}{2t}}\bigr)\le \dfrac{C}{R}t^{-1}.
	\end{aligned}
\end{equation}

On the other hand, noting that
\begin{equation}
\dfrac{1}{-\eta+i\zeta}=\dfrac{1}{i\zeta}-\eta\left(\dfrac{1}{\zeta^2+\eta^2}+\dfrac{1}{i\zeta}\dfrac{\eta}{\zeta^2+\eta^2}\right).
\end{equation}
Thus, since $|x|>Ct$, we have
\begin{equation}\label{est:gamma2}
\begin{aligned}
	\left|\lim_{K\to +\infty} \int_{\gamma_2}\dfrac{e^{i(x-\alpha_jt)\xi}}{i\xi}\,d\xi\right |&=\left| \int_{R}^{+\infty}\dfrac{e^{ix\zeta-i\alpha_j\zeta t-\frac{|x|^2}{t}+x\alpha_j}}{-\frac xt+i\zeta}\,d\zeta\right|\\
	&\le e^{-\frac{|x|^2}{2t}}\left|\int_{R}^{+\infty}e^{ix\zeta}\left(\dfrac{1}{i\zeta}-\dfrac xt \left(\dfrac{1}{\zeta^2+\frac {|x|^2}{t^2}}+\dfrac{1}{i\zeta}\dfrac{\frac xt}{\zeta^2+\frac {|x|^2}{t^2}}\right)\right)\,d\zeta\right|\\
	&\le Ce^{-\frac{|x|^2}{2t}}\left(\left|\int_{R}^{+\infty}\dfrac{e^{ix\zeta}}{i\zeta}\,d\zeta\right|+\left(\dfrac{|x|}{t}+\dfrac{|x|^2}{t^2}\right)\int_{R}^{+\infty}\dfrac{1}{\zeta^2}\,d\zeta\right)\\
	&\le Ce^{-\frac{|x|^2}{2t}}\dfrac{|x|^2}{2t}\left(\dfrac{t}{|x|}+\dfrac{1}{|x|}+\dfrac{1}{t}\right)\le Ce^{-\delta t}.
\end{aligned}
\end{equation}

Similarly, we consider $\gamma_3(s)=K+i\frac{x}{t}(1-s)$ for $s\in[0,1]$, we have
\begin{equation}\label{est:gamma31}
	\left|\lim_{K\to +\infty} \int_{\gamma_3}\dfrac{e^{i(x-\alpha_jt)\xi}}{i\xi}\,d\xi\right |= \left|\lim_{K\to +\infty} \int_{0}^{1}\dfrac{e^{i(x-\alpha_jt)K+x\alpha_j(1-s)-\frac{|x|^2}{t}(1-s)}}{K+i\frac{x}{t}(1-s)}\,\dfrac{x}{t}ds\right|.
\end{equation}
On the other hand, noting that for a fixed $K$, we have
\begin{equation}\label{est:gamma32}
\begin{aligned}
	\left| \int_{0}^{1}\dfrac{e^{i(x-\alpha_jt)K+x\alpha_j(1-s)-\frac{|x|^2}{t}(1-s)}}{K+i\frac{x}{t}(1-s)}\,\dfrac{x}{t}ds\right|&\le \dfrac{C}{K}\int_{0}^{1}\left(\dfrac{|x|}{t}+\dfrac{|x|^2}{t^2}\right)e^{-\frac{|x|^2}{2t}(1-s)}\,ds\\
	&=\dfrac{C}{K}\left(\dfrac{1}{|x|}+\dfrac{1}{t}\right)e^{-\frac{|x|^2}{2t}}\bigl(e^{\frac{|x|^2}{2t}}-1\bigr)\le \dfrac{C}{K}t^{-1}.
\end{aligned}
\end{equation}
One deduces that
\begin{equation}\label{est:gamma33}
	\lim_{K\to +\infty} \int_{0}^{1}\dfrac{e^{i(x-\alpha_jt)K+x\alpha_j(1-s)-\frac{|x|^2}{t}(1-s)}}{K+i\frac{x}{t}(1-s)}\,\dfrac{x}{t}ds=0.
\end{equation}
Hence, it implies that
\begin{equation}\label{est:gamma34}
	\left|\lim_{K\to +\infty} \int_{\gamma_3}\dfrac{e^{i(x-\alpha_jt)\xi}}{i\xi}\,d\xi\right |=0.
\end{equation}

Finally, one can estimate $H_1$ similarly by substituting $R$ and $K$ by $-R$ and $-K$ respectively. Therefore, from \eqref{eq:I_1^a |x| large}, \eqref{eq:H}, \eqref{est:gamma1}, \eqref{est:gamma2} and \eqref{est:gamma34}, one obtains
\begin{equation*}
	\bigl\| \mathcal F_j^{-1}(I_1)\bigr\|_{L^{\infty}}\le Ce^{-\frac12\theta t}
\end{equation*}
for $|x|>Ct$ where $C$ large enough since $\textrm{\normalfont Re}\,(\beta_{j\ell})\ge \theta>0$ and $\Theta_{j\ell}^{(0)}$ is a nilpotent matrix for all $j\in\{1,\dots,s\}$ and $\ell\in\{1,\dots,s_j\}$ by Proposition \ref{prop:high frequency}.

Therefore, it implies that\begin{equation*}
	\bigl\| \mathcal F^{-1}(I_1)\bigr\|_{L^{\infty}}\le C\sum_{j=1}^s\bigl\| \mathcal F_j^{-1}(I_1)\bigr\|_{L^{\infty}} \le Ce^{-\frac12\theta t}.
\end{equation*}

We estimate for $\mathcal F^{-1}(I_2)$ and $\mathcal F^{-1}(I_3)$ where $I_2$ and $I_3$ are in \eqref{eq:I_2 large} and \eqref{eq:I_3 large} respectively. Since $\mathcal F^{-1}:L^1\to L^\infty$, one has
\begin{equation}
	\bigl\| \mathcal F^{-1}(I_{2,3})\bigr\|_{L^{\infty}}\le C\bigl\| I_{2,3}\bigr\|_{L^1}.
\end{equation}
Hence, we only need to estimate $I_2$ and $I_3$ in $L^1$. 

From \eqref{eq:I_2 large}, we have
\begin{equation*}
	|I_2|\le C\sum_{j,\ell=1}^{s,s_j}e^{-\textrm{\normalfont Re}\,(\beta_{j\ell})t}e^{\bigl|\Theta_{j\ell}^{(0)}\bigr|t}|\xi|^{-2}t.
\end{equation*}
Thus, we obtain
\begin{equation*}
	\bigl\|I_2\bigr\|_{L^1}\le Ce^{-\frac12\theta t}, \qquad\textrm{for }t\ge 1.
\end{equation*}

From \eqref{eq:I_3 large}, we have
\begin{equation*}
	|I_3|\le \sum_{j,\ell=1}^{s,s_j}e^{-\textrm{\normalfont Re}\,(\beta_{j\ell})t}\left| e^{-\Theta_{j\ell}^{(0)}t+\mathcal O(|\xi|^{-1})t}-e^{-\Theta_{j\ell}^{(0)}t}-e^{-\Theta_{j\ell}^{(0)}t}\mathcal O(|\xi|^{-1})t\right|.
\end{equation*}
Then, by Lemma \ref{lem:nilpotent}, we obtain
\begin{equation*}
	|I_3|\le \sum_{j,\ell=1}^{s,s_j}e^{-\textrm{\normalfont Re}\,(\beta_{j\ell})t}e^{\varepsilon' t+C|\xi|^{-1}t}|\xi|^{-2}t^2,
\end{equation*}
where $\varepsilon'$ is small enough. Therefore, since $|\xi|$ large enough, we have
\begin{equation*}
	\bigl\|I_3\bigr\|_{L^1}\le Ce^{-\frac12\theta t},\qquad \textrm{for }t\ge 1.
\end{equation*}

Thus, we deduces
\begin{equation*}
	\bigl\| \mathcal F^{-1}(I)\bigr\|_{L^{\infty}}\le \bigl\| \mathcal F^{-1}(I_1)\bigr\|_{L^{\infty}}+\bigl\| \mathcal F^{-1}(I_2)\bigr\|_{L^{\infty}}+\bigl\| \mathcal F^{-1}(I_3)\bigr\|_{L^{\infty}}\le  Ce^{-\frac12\theta t}, \quad \textrm{for }t\ge 1.
\end{equation*}

We now estimate $\mathcal F^{-1}(J)$ where $J$ is given by \eqref{eq:J large}. From \eqref{eq:J large}, one has $J=J_1+J_2$ where
\begin{equation}
	J_1:=\sum_{j,\ell=1}^{s,s_j}\dfrac{e^{-i\alpha_j\xi t}}{i\xi}e^{-\beta_{j\ell}t}e^{-\Theta_{j\ell}^{(0)}t+\mathcal O(|\xi|^{-1})t}M,
\end{equation}
where $M$ is the coefficient associated with the term $(i\xi)^{-1}$ in $\mathcal O(|\xi|^{-1})$, and
\begin{equation*}
	J_2:=\sum_{j,\ell=1}^{s,s_j}e^{(-i\alpha_j\xi-\beta_{j\ell})t}e^{-\Theta_{j\ell}^{(0)}t+\mathcal O(|\xi|^{-1})t}\mathcal O(|\xi|^{-2}).
\end{equation*}
Then, we can estimate $\mathcal F^{-1}(J_1)$ as the case of $\mathcal F^{-1}(I_1)$ and we can estimate $\mathcal F^{-1}(J_2)$ as the case of $\mathcal F^{-1}(I_2)$ and $\mathcal F^{-1}(I_3)$. Thus, we deduces
\begin{equation*}
	\bigl\| \mathcal F^{-1}(J)\bigr\|_{L^{\infty}}\le \bigl\| \mathcal F^{-1}(J_1)\bigr\|_{L^{\infty}}+\bigl\| \mathcal F^{-1}(J_2)\bigr\|_{L^{\infty}}\le Ce^{-\frac12\theta t}, \qquad \textrm{for }t\ge 1.
\end{equation*}

Therefore, we conclude
\begin{equation*}
	\bigl\| \mathcal F^{-1}(\hat G-\hat V)\bigr\|_{L^{\infty}}\le \bigl\| \mathcal F^{-1}(I)\bigr\|_{L^{\infty}}+\bigl\| \mathcal F^{-1}(J)\bigr\|_{L^{\infty}} \le Ce^{-\frac12\theta t}
\end{equation*}
for $t\ge 1$ and the proof is done.
\end{proof}

\section{Multiplier estimates}\label{sec:multiplier estimates}
This section provides some useful Fourier multiplier estimates by recalling the Young inequality.
\begin{lemma}[Young's inequality]
For all $(p,q,r)\in[1,\infty]^3$ such that $\frac1q-\frac1p=1-\frac1r$ and $(f,g)\in L^r\times L^q$, we have $f*g\in L^p$
and $\|f*g\|_{L^p}\le\|f\|_{L^r}\|g\|_{L^q}$.
\end{lemma}

We thus obtain the follows.
\subsection{Case $|x|\le Ct$}
Let $\chi_1$ and $\chi_3$ be cutoff functions defined on $[-\varepsilon,\varepsilon]$ and $(-\infty,-R]\cup[R,+\infty)$ respectively for $\varepsilon$ small and $R$ large such that $|\chi_{1,2}|\le 1$. Let $\chi_2:=1-\chi_1-\chi_3$, we introduce the multipliers
\begin{equation*}
	m_j:=\chi_j(\hat G-\hat K-\hat V),\qquad j=1,2,3.
\end{equation*}
The following holds.
\begin{proposition}\label{prop:multiplier 1}
For $r\in[1,\infty]$, $m_j\in M_r$ with
\begin{equation}
	\|m_j\|_{M_r}\le Ct^{-\frac12},\qquad j=1,2,3\textrm{ and }t\ge 1.
\end{equation}
\begin{proof}
We begin with $m_1$. For $|\xi|\le \varepsilon$, we have $\hat G-\hat K=I_1+I_2+J$ where $I_1,I_2$ are in \eqref{eq:I_1 small}, \eqref{eq:I_2 small} respectively and $J=\hat G_2$ as in \eqref{eq:G2 small}. We then have
\begin{equation*}
	\mathcal F^{-1}(\chi_1I_1)(x,t)=\sum_{j,\ell=1}^{h,h_j}\int_{-\varepsilon}^{\varepsilon}\chi_1(\xi)e^{i(x-c_jt)\xi-d_{j\ell}\xi^2t}\left(e^{-N_{j\ell}^{(0)}\xi^2t+\mathcal O(|\xi|^3)t}-e^{-N_{j\ell}^{(0)}\xi^2t}\right)P_{j\ell}^{(0)} d\xi.
\end{equation*}
For $j\in\{1,\dots,h\}$ and $\ell\in\{1,\dots,h_j\}$, let $z=e^{i\phi/2}\xi$ where $\phi=\textrm{\normalfont arg}\,(d_{j\ell})\in (-\pi/2,\pi/2)$ since $\textrm{\normalfont Re}\,(d_{j\ell})>0$, one obtains
\begin{align*}
	\mathcal F^{-1}(\chi_1I_1)(x,t)&=\sum_{j,\ell=1}^{h,h_j}\int_{\gamma}\chi_1(e^{-i\phi/2}z)e^{i(x-c_jt)e^{-i\phi/2}z-|d_{j\ell}|z^2t}\\
	&\hskip2cm\cdot\left(e^{-N_{j\ell}^{(0)}e^{-i\phi}z^2t+\mathcal O(|e^{-i\phi/2}z|^3)t}-e^{-N_{j\ell}^{(0)}e^{-i\phi}z^2t}\right)P_{j\ell}^{(0)} e^{-i\phi/2}dz,
\end{align*}
where $\gamma:=\{z\in\mathbb C:z=e^{i\phi/2}\xi,\xi\in[-\varepsilon,\varepsilon]\}$. Then, we will estimate for each summand by letting $\eta:=\min\bigl\{\frac{|x-c_jt|}{2|d_{j\ell}|t},\frac{\varepsilon}{2}\bigr\}$. Since the integrand is holomorphic, we can change the path of the integral from $\gamma$ to $\gamma:=\gamma_1\cup \gamma_2\cup\gamma_3$ in the complex plane where
\begin{equation}\label{gamma1}
	\gamma_1:=\left\{-\varepsilon e^{i\phi/2}+i\textrm{\normalfont sgn}(x-c_jt)\eta e^{-i\phi/2}s:s\in[0,1]\right\},
\end{equation}
\begin{equation}\label{gamma2}
	\gamma_2:=\left\{ \zeta e^{i\phi/2}+i\textrm{\normalfont sgn}(x-c_jt)\eta e^{-i\phi/2}:\zeta\in[-\varepsilon,\varepsilon]\right\}
\end{equation}
and
\begin{equation}\label{gamma3}
	\gamma_3:=\left\{\varepsilon e^{i\phi/2}+i\textrm{\normalfont sgn}(x-c_jt)\eta e^{-i\phi/2}(1-s):s\in[0,1]\right\}.
\end{equation}

On the other hand, we have
\begin{equation}
	\left|e^{i(x-c_jt)e^{-i\phi/2}z-|d_{j\ell}|z^2t}\right|=e^{-(x-c_jt)\bigl(\cos(\phi/2)\textrm{\normalfont Im}\,z-\sin(\phi/2)\textrm{\normalfont Re}\,z\bigr)}e^{-|d_{j\ell}|(\textrm{\normalfont Re}\,z-\textrm{\normalfont Im}\,z)(\textrm{\normalfont Re}\,z+\textrm{\normalfont Im}\,z)t}.
\end{equation}
Moreover, $|\chi_1|\le 1$ and similarly to before, by Lemma \ref{lem:nilpotent}, since $N_{j\ell}^{(0)}$ is nilpotent and since $|z|=|\xi|\le \varepsilon$ small, we have
\begin{equation}
\begin{aligned}
	\left|e^{-N_{j\ell}^{(0)}e^{-i\phi}z^2t+\mathcal O(|e^{-i\phi/2}z|^3)t}-e^{-N_{j\ell}^{(0)}e^{-i\phi}z^2t}\right|&\le C|z|^3te^{\varepsilon'|z|^2t+C|z|^3t}\\
	&\le C(|\textrm{\normalfont Re}\,z|+|\textrm{\normalfont Im}\,z|)^3te^{\varepsilon''(|\textrm{\normalfont Re}\,z|+|\textrm{\normalfont Im}\,z|)^2t},
\end{aligned}
\end{equation}
where $\varepsilon',\varepsilon''$ can be chosen as small as one needs.

Thus, for $z\in\gamma_1$, we have
\begin{equation*}
\begin{aligned}
	\textrm{\normalfont Re}\,z&=-\varepsilon\cos(\phi/2)+\textrm{\normalfont sgn}(x-c_jt)\eta\sin(\phi/2)s,\\
	\textrm{\normalfont Im}\,z&=-\varepsilon \sin(\phi/2) +\textrm{\normalfont sgn}(x-c_jt)\eta\cos(\phi/2)s.
\end{aligned}
\end{equation*}
We then obtain from $\cos(\phi)>0$, $\eta^2s^2\le \varepsilon^2/2$ for $s\in[0,1]$ and $|z|\le \varepsilon$ that for some $\delta>0$, one has
\begin{equation}\label{est:gamma1small}
	\left|\int_{\gamma_1}\right|\le C\int_0^1e^{-|x-c_jt|\eta \cos (\phi)s}e^{-|d_{j\ell}|\cos(\phi)(\varepsilon^2-\eta^2s^2)t}e^{\varepsilon''\varepsilon^2 t}\varepsilon^3tds\le  Ce^{-\delta t}.
\end{equation}

For $z\in \gamma_2$, we have
\begin{equation*}
\begin{aligned}
	\textrm{\normalfont Re}\,z&=\zeta\cos(\phi/2)+\textrm{\normalfont sgn}(x-c_jt)\eta\sin(\phi/2),\\
	\textrm{\normalfont Im}\,z&=\zeta \sin(\phi/2) +\textrm{\normalfont sgn}(x-c_jt)\eta\cos(\phi/2).
\end{aligned}
\end{equation*}
Hence, one has
\begin{equation}
	\left|\int_{\gamma_2}\right|\le C\int_{-\varepsilon}^{\varepsilon}e^{-|x-c_jt|\eta \cos (\phi)}e^{-|d_{j\ell}|\cos(\phi)(\zeta^2-\eta^2)t}e^{\varepsilon''(\zeta^2+2|\zeta||\eta|+|\eta|^2) t}(|\zeta|+|\eta|)^3td\zeta.
\end{equation}
If $\eta=\frac{|x-c_jt|}{2|d_{j\ell}|t}$, then since $|\zeta|\le \varepsilon$ small and $\varepsilon''$ small enough, for some $c>0$, we have
\begin{equation}\label{est:gamma2asmall}
	\begin{aligned}
		\left|\int_{\gamma_2}\right|&\le C\int_{-\varepsilon}^{\varepsilon}e^{-\frac{|x-c_jt|^2}{|d_{j\ell}|t} \cos (\phi)}e^{\frac{|x-c_jt|^2}{2|d_{j\ell}|t}\cos(\phi)}e^{-|d_{j\ell}|\cos(\phi)\zeta^2t}e^{\varepsilon''\zeta^2t+\varepsilon''|\zeta|\frac{|x-c_jt|}{|d_{j\ell}|t}+\varepsilon''\frac{|x-c_jt|^2}{4|d_{j\ell}|^2t}}\\
		&\hskip3cm\cdot \left(|\zeta|^3t+3|\zeta|^2\dfrac{|x-c_jt|}{2|d_{j\ell}|}+3|\zeta|\dfrac{|x-c_jt|^2}{4|d_{j\ell}|^2t}+\dfrac{|x-c_jt|^3}{8|d_{j\ell}|^3t^2}\right)d\zeta\\
		&\le C \sum_{k=0}^3e^{-\frac{|x-c_jt|^2}{8|d_{j\ell}|t} \cos (\phi)}\left(\dfrac{|x-c_jt|}{\sqrt{t}}\right)^{k}\int_{-\varepsilon}^{\varepsilon}e^{-\frac12|d_{j\ell}|\cos(\phi)\zeta^2t}|\zeta|^{3-k}t^{1-\frac k2}d\zeta\\
		&\le Ct^{-1}e^{-\frac{|x-c_jt|^2}{c|d_{j\ell}|t}}.
	\end{aligned}
\end{equation}
If $\eta=\varepsilon/2$, then $|x-c_jt|\ge \varepsilon|d_{j\ell}|t$ by the definition of $\eta$ and we have
\begin{equation}\label{est:gamma2bsmall}
\begin{aligned}
	\left|\int_{\gamma_2}\right|&\le C\int_{-\varepsilon}^{\varepsilon}e^{-|x-c_jt|\eta \cos (\phi)}e^{-|d_{j\ell}|\cos(\phi)(\zeta^2-\eta^2)t}e^{\varepsilon''(\zeta^2+2|\zeta||\eta|+|\eta|^2) t}(|\zeta|+|\eta|)^3td\zeta\\
	&\le Ce^{-\varepsilon^2|d_{j\ell}|\cos (\phi)t}e^{\frac14\varepsilon^2|d_{j\ell}|\cos(\phi)t}e^{\varepsilon''\varepsilon^2 t}\int_{-\varepsilon}^{\varepsilon}e^{-|d_{j\ell}|\cos(\phi)\zeta^2t}\left(|\zeta|+\dfrac{\varepsilon}{2}\right)^3td\zeta\le Ce^{-\delta t},
\end{aligned}
\end{equation}
for some $\delta>0$ since $\varepsilon''$ can be chosen small enough.

For $z\in\gamma_3$, we have
\begin{equation*}
\begin{aligned}
	\textrm{\normalfont Re}\,z&=\varepsilon\cos(\phi/2)+\textrm{\normalfont sgn}(x-c_jt)\eta\sin(\phi/2)(1-s),\\
	\textrm{\normalfont Im}\,z&=\varepsilon \sin(\phi/2) +\textrm{\normalfont sgn}(x-c_jt)\eta\cos(\phi/2)(1-s).
\end{aligned}
\end{equation*}
Thus, similarly to $\gamma_1$, for some $\delta>0$, one has
\begin{equation}\label{est:gamma3small}
	\left|\int_{\gamma_3}\right|\le C\int_0^1e^{-|x-c_jt|\eta \cos (\phi)(1-s)}e^{-|d_{j\ell}|\cos(\phi)(\varepsilon^2-\eta^2(1-s)^2)t}e^{\varepsilon''\varepsilon^2 t}\varepsilon^3tds\le  Ce^{-\delta t}.
\end{equation}

Therefore, from \eqref{est:gamma1small}, \eqref{est:gamma2asmall}, \eqref{est:gamma2bsmall}, \eqref{est:gamma3small} and the fact that $e^{-\delta t}\le Ct^{-1}e^{-\frac{|x-c_jt|^2}{c|d_{j\ell}|t}}$ since $|x|\le Ct$ and $t\ge 1$, we obtain
\begin{equation*}
	\left|\mathcal F^{-1}(\chi_1I_1)(x,t)\right|\le Ct^{-1}e^{-\frac{|x-c_jt|^2}{c|d_{j\ell}|t}},\qquad t\ge 1.
\end{equation*}
By the same way for $\mathcal F^{-1}(\chi_1I_2)$ and $\mathcal F^{-1}(\chi_1J)$, we also have
\begin{equation*}
	\left|\mathcal F^{-1}(\chi_1(I_2+J))(x,t)\right|\le Ct^{-1}e^{-\frac{|x-c_jt|^2}{c|d_{j\ell}|t}},\qquad t\ge 1.
\end{equation*}
Hence, for $r\in[1,\infty]$, by the Young inequality and since $m_1=\chi_1(I_1+I_2+J)$, it follows that
\begin{equation*}
	\|m_1\|_{M_r}=\sup_{\|f\|_{L^r}=1}\bigl\|\mathcal F^{-1}(m_1)*f\bigr\|_{L^r}\le \bigl\|\mathcal F^{-1}(m_1)\bigr\|_{L^1}\le Ct^{-\frac12},\qquad t\ge 1.
\end{equation*}

We consider $m_2$. Since $|\chi_2(\xi)|,|e^{ix\xi }|\le 1$ for $\xi\in \mathbb R$, we have
\begin{equation}
	\bigl|\mathcal F^{-1}(m_2)(x,t)\bigr|\le \int_{\varepsilon\le |\xi|\le R}\bigl(|\hat G(\xi,t)|+|\hat K(\xi,t)|+|\hat V(\xi,t)|\bigr)d\xi\le Ce^{-\delta t}
\end{equation}
for some $\delta>0$ due to \eqref{est:Kintermediate}, \eqref{est:Vintermediate} and the fact that $|\hat G(\xi,t)|\le e^{-\frac{\theta|\xi|^2}{1+|\xi|^2}t}$ on $\varepsilon\le |\xi|\le R$ for $\theta>0$.

Thus, one has
\begin{equation*}
	\left|\mathcal F^{-1}(m_2)(x,t)\right|\le Ct^{-1}e^{-\frac{|x-c_jt|^2}{c|d_{j\ell}|t}},\qquad t\ge 1.
\end{equation*}

Hence, for $r\in[1,\infty]$, by the Young inequality, it follows that
\begin{equation*}
	\|m_2\|_{M_r}=\sup_{\|f\|_{L^r}=1}\bigl\|\mathcal F^{-1}(m_2)*f\bigr\|_{L^r}\le \bigl\|\mathcal F^{-1}(m_2)\bigr\|_{L^1}\le Ct^{-\frac12},\qquad t\ge 1.
\end{equation*}

Finally, we consider $m_3$. Based on the decomposition $\hat G-\hat V=I+J$ where $I$ is defined as $I_1$ in \eqref{eq:I_1 large} and $J$ is the remainder, from \eqref{est:I_1 large} and a same treatment for $J$, we obtain for $|x|\le Ct$ that
\begin{equation*}
	\left|\mathcal F^{-1}(\chi_3(\hat G-\hat V))(x,t)\right|\le C\sum_{j=1}^ste^{-\delta t}|x-\alpha_jt|+Ce^{-\delta t}\le Ct^{-1}e^{-\frac{|x-c_jt|^2}{c|d_{j\ell}|t}},\qquad t\ge 1
\end{equation*}
for some $\delta>0$.

Moreover, from \eqref{est:K large}, there is a $\theta>0$ such that
\begin{equation*}
	\left|\mathcal F^{-1}(\chi_3\hat K)(x,t)\right|\le e^{-\frac12\theta R^2t}\int_{|\xi|\ge R}e^{-\frac14\theta |\xi|^2t}d\xi\le  Ce^{-\delta t}\le Ct^{-1}e^{-\frac{|x-c_jt|^2}{c|d_{j\ell}|t}},\qquad t\ge 1
\end{equation*}
for some $\delta>0$.

Hence, for $r\in[1,\infty]$, by the Young inequality and $m_3=\chi_3(\hat G-\hat V-\hat K)$, it follows that
\begin{equation*}
	\|m_3\|_{M_r}=\sup_{\|f\|_{L^r}=1}\bigl\|\mathcal F^{-1}(m_3)*f\bigr\|_{L^r}\le \bigl\|\mathcal F^{-1}(m_3)\bigr\|_{L^1}\le Ct^{-\frac12},\qquad t\ge 1.
\end{equation*}
We finish the proof.
\end{proof}
\end{proposition}

\subsection{Case $|x|>Ct$}
We introduce the multipliers
\begin{equation*}
	m^1:=\hat G-\hat V,\qquad \textrm{and} \qquad m^2:=\hat K.
\end{equation*}
The following holds.
\begin{proposition}\label{prop:multiplier 2}
For $r\in[1,\infty]$, we have $m^j\in M_r$ with
\begin{equation}
	\|m^j\|_{M_r}\le Ct^{-\frac12},\qquad j=1,2\textrm{ and }t\ge 1.
\end{equation}
\begin{proof}
We estimate the $L^1$-norm of $\mathcal F^{-1}(m^1)$. We have
\begin{equation}\label{multiplier:1 upper}
	\mathcal F^{-1}(m^1)(x,t)=\lim_{R\to +\infty}\int_{-R}^{R}e^{ix\xi}\bigl(\hat G(\xi,t)-\hat V(\xi,t)\bigr)d\xi.
\end{equation}
On the other hand, noting that the solution $\hat G$ to \eqref{eq:fundamental system} is written as $\hat G(\xi,t)=e^{E(i\xi)t}$ and thus $\hat G$ is an entire function on the complex plane since $E(i\xi)=-(B+i\xi A)$. Moreover, due to the formula of $\hat V$ in \eqref{eq:exponential decay kernel}, $\hat V$ is also holomorphic on the complex plane. Thus, by considering $\xi=\zeta+i\eta\in\mathbb C$, one can change the path of the integral in \eqref{multiplier:1 upper} from $\{(\zeta,0):\zeta\textrm{ from }-R \textrm{ to } R\}$ to the path $\gamma:=\gamma_1\cup \gamma_2\cup\gamma_3$ in the complex plane where
\begin{equation}
	\gamma_1:=\left\{(\zeta,\eta):\zeta=-R,\eta\textrm{ from }0 \textrm{ to } \frac xt\right\},
\end{equation}
\begin{equation}
	\gamma_2:=\left\{(\zeta,\eta):\zeta\textrm{ from } -R \textrm{ to } R,\eta=\frac xt\right\}
\end{equation}
and
\begin{equation}
	\gamma_3:=\left\{(\zeta,\eta):\zeta=R,\eta\textrm{ from }\frac xt \textrm{ to } 0\right\}.
\end{equation}
Furthermore, since $R$ and $|x|/t$ large, along these curves, the solution $\hat G$ has the representation of the high frequency case \eqref{eq:G large}. Therefore, by the same computation as in \eqref{est:gamma1}-\eqref{est:gamma34} and letting $R\to +\infty$, we obtain
\begin{equation}
	\bigl| \mathcal F^{-1}(m^1)(x,t)\bigr|\le Ce^{-\frac{|x|^2}{ct}}\le Ct^{-1}e^{-\frac{|x|^2}{2c t}}
\end{equation}
for some $c,C>0$ since $e^{-\frac{|x|^2}{2ct}}\le e^{-C^2t}\le t^{-1}$ due to the fact that $|x|>Ct$ with $C$ large enough. Hence, we obtain
\begin{equation}
	\bigl\| \mathcal F^{-1}(m^1) \bigr\|_{L^1}\le Ct^{-\frac12}.
\end{equation}
Thus, by the Young inequality, for $r\in[1,\infty]$, we have
\begin{equation}
	\bigl\| m^1\bigr\|_{M_r}=\sup_{\|f\|_{L^r}=1}\bigl\|\mathcal F^{-1}(m^1)*f  \bigr\|_{L^r}\le \bigl\|\mathcal F^{-1}(m^1) \bigr\|_{L^1}\le Ct^{-\frac12}.
\end{equation}
The estimate for $m^2$ are similar and the proof is done.
\end{proof}
\end{proposition}

\section{Symmetry}\label{sec:symmetry}
We will discuss about the conditions {\bf C'} and {\bf S} in order to increase the decay rate of the solution to the system \eqref{eq:original hyperbolic system}. Recalling the matrices $C$ and $D$ as in \eqref{eq:reduced system} and \eqref{eq:matrix D} respectively.

\begin{lemma}\label{lem:analytic eigenvalue}
If the condition {\bf C'} holds, then there are $m$ distinct eigenvalues of $E(i\xi)=-(B+i\xi A)$ converging to $0$ as $|\xi|\to 0$ and they are expanded analytically, where $m=\dim\ker(B)$. The approximation of the $j$-th eigenvalue has the form
\begin{equation}\label{eq:eigenvalues of E}
	\lambda_j(i\xi)=-ic_j\xi-d_j\xi^2+\mathcal O(|\xi|^3),\qquad |\xi|\to 0,
\end{equation}
where $c_j\in\sigma(C)$ considered in $\ker(B)$ and $d_{j}\in\sigma\bigl(P_j^{(0)}DP_j^{(0)}\bigr)$ considered in $\ker(C-c_jI)$ with $P_j^{(0)}$ the eigenprojection associated with $c_j$ for $j\in\{1,\dots,m\}$.
\end{lemma}
\begin{proof}
This is just a consequence of Proposition \ref{prop:low frequency 1}. Indeed, from \eqref{eq:expansion of lambdajell}, the approximation of the eigenvalues of $E$ converging to $0$ as $|\xi|\to 0$ is
\begin{equation*}
	\lambda_{j\ell}(i\xi)=-ic_j\xi-d_{j\ell} \xi^2+{\scriptstyle\mathcal O}(|\xi|^2),\qquad |\xi|\to 0,
\end{equation*}
where $c_j\in\sigma(C)$ considered in $\ker(B)$ and $d_{j\ell}\in\sigma\bigl(P_j^{(0)}DP_j^{(0)}\bigr)$ considered in $\ker(C-c_jI)$ for $\ell=1,\dots,h_j$ with $P_j^{(0)}$ the eigenprojection associated with $c_j$ for $j\in\{1,\dots,h\}$. Noting that $h_j$ is the cardinality of the spectrum of $P_j^{(0)}DP_j^{(0)}$ considered in $\ker(C-c_jI)$ for $j\in\{1,\dots,h\}$ and $h$ is the cardinality of the spectrum of $C$ considered in $\ker(B)$.

On the other hand, since the condition {\bf C'} holds, $h=m$ where $m=\dim\ker(B)$. Moreover, also by the condition {\bf C'}, one deduces that $c_j$ is simple for all $j\in\{1,\dots,m\}$. Thus, $\dim\ker(C-c_jI)=1$ and therefore $h_j=1$ for $j\in\{1,\dots,m\}$. It implies that there is only one $d_j:=d_{j1}\in \bigl(P_j^{(0)}DP_j^{(0)}\bigr)$ considered in $\ker(C-c_jI)$ for each $j\in\{1,\dots,m\}$. Moreover, $d_j$ is also simple, and thus, one can continue the reduction process as in the proof of Proposition \ref{prop:low frequency 1}. Furthermore, due to the simplicity of the coefficients in the expansion of $\lambda_j$ provided $c_j$ is simple and the reduction process, there is no splitting in the expansion of the eigenvalues $\lambda_j$ {\it i.e.} the eigenvalues $\lambda_j$ can be expanded analytically for $j\in\{1,\dots,m\}$ and the proof is done.
\end{proof}

Let $p(\lambda,\kappa):=\det (E(\kappa)-\lambda I)$ be the {\it dispersion polynomial} associated with $E(\kappa)=-(B+\kappa A)$, where $\lambda,\kappa\in\mathbb C$.
\begin{lemma}\label{lem:symmetry}
If the condition {\bf S} holds, then $p(\lambda,-\kappa)=p(\lambda,\kappa)$
for any $\lambda,\kappa\in\mathbb{C}$.
\end{lemma}
\begin{proof}
For $q(\lambda,\kappa):=p(\lambda,-\kappa)$, there holds $p(\lambda,\kappa)=0$
if and only if $q(\lambda,\kappa)=0$. Indeed, if the couple $(\lambda,\kappa)$
is such that $p(\lambda,\kappa)=0$ if and only if there exist a nonzero
vector $u$ such that 
\begin{equation*}
	\bigl(\lambda\,I+\kappa A +B\bigr)u=0.
\end{equation*}
In such a case, setting $v=S^{-1}u$, there also holds
\begin{equation*}
	\begin{aligned}
		0 & =S^{-1}\bigl(\lambda\,I+ \kappa A +B\bigr)Sv=S^{-1}\bigl(\lambda\,S+\kappa A  S+BS\bigr)v\\
			 & =S^{-1}\bigl(\lambda\,S-\kappa SA+SB\bigr)v=\bigl(\lambda\,I- \kappa A +B\bigr)v.
	\end{aligned}
\end{equation*}
Hence, $q(\lambda,\kappa)=0$. The other implication can be proved
in the same way.

For fixed $\kappa$, the polynomials $p$ and $q$ have both degree
$n$ in $\lambda$ with principal term $\lambda^n$. Hence, there
exist $\lambda_1^p,\dots,\lambda_n^p$ and $\lambda_1^q,\dots,\lambda_n^q$
with $\lambda_i^{p,q}=\lambda_i(\kappa)$ such that 
\begin{equation*}
	p(\lambda,\kappa)=\prod_{k=1}^n\bigl(\lambda-\lambda_k^p(\kappa)\bigr)\quad\textrm{and}\quad 
	q(\lambda,\kappa)=\prod_{k=1}^n\bigl(\lambda-\lambda_k^q(\kappa)\bigr)
\end{equation*}
Since $p$ and $q$ have the same zero-set, for any $k\in\{1,\dots,n\}$
there exists $j$ such that $\lambda_k^q=\lambda_j^p$. As
a consequence, $p\equiv q$.
\end{proof}

\begin{corollary}\label{cor:symmetry 0 group}
If the conditions {\bf C'} and {\bf S} hold, then there are $m$ analytic distinct eigenvalues of $E$ converging to $0$ as $|\xi|\to 0$ and the $j$-th eigenvalue has the aproximation
\begin{equation}
	\lambda_j(i\xi)=-ic_j\xi-d_j\xi^2+\mathcal O(|\xi|^4),\qquad |\xi|\to 0,
\end{equation}
where $c_j\in\sigma(C)$ considered in $\ker(B)$ and $d_{j}\in\sigma\bigl(P_j^{(0)}DP_j^{(0)}\bigr)$ considered in $\ker(C-c_jI)$ with $P_j^{(0)}$ the eigenprojection associated with $c_j$ for $j\in\{1,\dots,m\}$.
\end{corollary}
\begin{proof}
From the proof of Lemma \ref{lem:analytic eigenvalue}, since $d_j$ is simple for all $j\in\{1,\dots,m\}$, one can continue the reduction process in the proof of Proposition \ref{prop:low frequency 1} and thus the formula \eqref{eq:eigenvalues of E} can be refined as
\begin{equation*}
	\lambda_j(i\xi)=-ic_j\xi-d_j\xi^2-e_j(i\xi)^3+\mathcal O(|\xi|^4),\qquad |\xi|\to 0,
\end{equation*}
where $e_j\in\sigma(M_j)$ considered in $\ker\bigl(P_j^{(0)}DP_j^{(0)}-d_jI\bigr)$ for some suitable matrix $M_j$ for $j\in\{1,\dots,m\}$.

By recalling the proof of Proposition \ref{prop:low frequency 1} and by Lemma \ref{lem:analytic eigenvalue} one more time, substituting $(-i\xi)$ into $i\xi$, there are $m$ analytic distinct eigenvalues of $E(-i\xi)$ converging to $0$ as $|\xi|\to 0$ such that the 
\begin{equation}\label{eq:eigenvalues of E'}
	\lambda_j(-i\xi)=-i(-c_j)\xi-d_j\xi^2-(-e_j)(i\xi)^3+\mathcal O(|\xi|^4),\qquad |\xi|\to 0,
\end{equation}
where $c_j,d_j$ and $e_j$ are already introduced as before.

On the other hand, since $\sigma(E(i\xi))\equiv\sigma(E(-i\xi))$ due to Lemma \ref{lem:symmetry}, one deduces that $\sigma(M_j)$ contains both $e_j$ and $-e_j$. Moreover, since $\dim\ker\bigl(P_j^{(0)}DP_j^{(0)}-d_jI\bigr)=1$, one concludes that $e_j=-e_j=0$. The proof is done.
\end{proof}

\begin{remark}
The nilpotent parts associated with $\lambda_j$ for $j\in\{1,\dots,m\}$ are zero since these eigenvalues are distinct and simple.

Moreover, for each $j\in\{1,\dots,m\}$, the total projection associated with $\lambda_j$ is itself the eigenprojection associated with $\lambda_j$ and has the expansion \eqref{eq:expansion of Pj} with $\zeta=i\xi$ {i.e.} we have
\begin{equation}
	P_j(i\xi)=P_j^{(0)}+i\xi P_j^{(1)}+\mathcal O(|\xi|^2),\qquad |\xi|\to 0,
\end{equation}
where $P_j^{(1)}$ can be computed by the formula \eqref{eq:Pj1} for $j\in\{1,\dots,m\}$. This is based on the fact that there is no splitting after the second step of the reduction process and the formula of $P_j^{(1)}$ is proved similarly to the proof of the formula \eqref{eq:P01} in the proof of Proposition \ref{prop:low frequency 1}.
\end{remark}

One sets the kernel
\begin{equation}
	\hat K^*(\xi,t):=\sum_{j=1}^me^{(-ic_j\xi-d_j\xi^2)t}\bigl(P_j^{(0)}+i\xi P_j^{(1)}\bigr).
\end{equation}
Then, the first estimate in \eqref{est:low} of Proposition \ref{prop:fundamental solution standard type} can be modified by
\begin{proposition}\label{prop:low refined}
	If the conditions {C'} and {\bf S} hold, for $r\in[1,\infty]$, one has
	\begin{equation}\label{est:low refined}
		\bigl\|\hat G-\hat K^*\bigr\|_{L^{r}}\le Ct^{-\frac12\frac{1}{r}-1},
	\end{equation}
for $|\xi|<\varepsilon$ small enough and $t\ge 1$.
\end{proposition}
\begin{proof}
	For $|\xi|<\varepsilon$, by Remark \ref{rem:low frequency 1}, Remark \ref{rem:low frequency 2} and Corollary \ref{cor:symmetry 0 group}, the solution to the system \eqref{eq:fundamental system} is given by $\hat G=\hat G_1+\hat G_2$ where $\hat G_2$ is given by \eqref{eq:G2 small} and
	\begin{equation*}
		\hat G_1(\xi,t)=\sum_{j=1}^me^{(-ic_j\xi-d_j\xi^2)t+\mathcal O(|\xi|^4)t}\bigl(P_j^{(0)}+i\xi P_j^{(1)}+\mathcal O(|\xi|^2)\bigr).
	\end{equation*}
Thus, similarly to the proof of the first estimate in \eqref{est:low}, we have $\hat G-\hat K^*=I_1+I_2+J$ where $J=\hat G_2$ and
\begin{align}
\label{eq:I_1 refined}	I_1&:=\sum_{j=1}^me^{(-ic_j\xi-d_j\xi^2)t}\left(e^{\mathcal O(|\xi|^4)t}-1\right)\bigl(P_j^{(0)}+i\xi P_j^{(1)}\bigr),\\
\label{eq:I_2 refined}	I_2&:=\sum_{j=1}^me^{(-ic_j\xi-d_j\xi^2)t+\mathcal O(|\xi|^4)t}\mathcal O(|\xi|^2).
\end{align}
Hence, similarly to before, there is a constant $c>0$ such that
\begin{equation*}
 	|I_1|\le Ce^{-c|\xi|^2t}|\xi|^4t \qquad\textrm{and}\qquad  |I_2|\le Ce^{-c|\xi|^2t}|\xi|^2.
\end{equation*}
Thus, together with \eqref{est:J small}, it implies that
\begin{equation*}
	\bigl\|\hat G-\hat K^*\bigr\|_{L^{r}}\le \bigl\|I_1\bigr\|_{L^{r}}+\bigl\|I_2\bigr\|_{L^{r}}+\bigl\|\hat J\bigr\|_{L^{r}}\le Ct^{-\frac12\frac{1}{r}-1},
\end{equation*}
for $|\xi|<\varepsilon$, $t\ge 1$ and $r\in[1,\infty]$. We finish the proof.
\end{proof}

Similarly, by recall the multipliers $m_j$ for $j=1,2,3$ and $m^j$ for $j=1,2$ with $\hat K$ is substituted by $\hat K^*$, we can also refine Proposition \ref{prop:multiplier 1} for $|x|\le Ct$ and Proposition \ref{prop:multiplier 2} for $|x|>Ct$.
\begin{proposition}[$|x|\le Ct$]\label{prop:multiplier 1 refined}
For $r\in[1,\infty]$, $m_j\in M_r$ with
\begin{equation}
	\|m_j\|_{M_r}\le Ct^{-1},\qquad j=1,2,3\textrm{ and }t\ge 1.
\end{equation}
\end{proposition}
\begin{proof}
Similarly to the proof of Proposition \ref{prop:multiplier 1}, we only need to consider $\mathcal F^{-1}(\chi_1I_1)$ and $\mathcal F^{-1}(\chi_1I_2)$ on $\gamma_2$ where $I_1,\,I_2$ are now given by \eqref{eq:I_1 refined}, \eqref{eq:I_2 refined} respectively and $\gamma_2$ is the same as \eqref{gamma2}. The others is bounded by $e^{-\delta t}$ for some $\delta >0$ and thus since $|x|\le Ct$, they are dominated by $t^{-\frac32}e^{-\frac{|x-c_jt|^2}{c|d_j|t}}$ for some $c>0$ and $t\ge 1$. 

Hence, noting that since $\left| e^{\mathcal O(|e^{-i\phi/2}z|^4)t}-1\right|\le C(|\textrm{\normalfont Re}\,z|+|\textrm{\normalfont Im}\,z|)^4te^{\varepsilon(|\textrm{\normalfont Re}\,z|+|\textrm{\normalfont Im}\,z|)^2t}$ for $z=e^{i\phi/2}\xi$ where $\xi\in[-\varepsilon,\varepsilon]$ and $\phi=\textrm{\normalfont arg}(d_j)\in (-\pi/2,\pi/2)$ for $j\in\{1,\dots,m\}$, on $\gamma_2$, we have
\begin{equation*}
\begin{aligned}
	\left|\mathcal F^{-1}(\chi_1I_1)\right|&\le C\sum_{j=1}^m\int_{-\varepsilon}^{\varepsilon}e^{-|x-c_jt|\eta \cos (\phi)}e^{-|d_j|\cos(\phi)(\zeta^2-\eta^2)t}e^{\varepsilon(\zeta+|\eta|)^2 t}(|\zeta|+|\eta|)^4td\zeta\\
	&\le C\sum_{j=1}^m \sum_{k=0}^4e^{-\frac{|x-c_jt|^2}{c|d_j|t} \cos (\phi)}\left(\dfrac{|x-c_jt|}{\sqrt{t}}\right)^{k}\int_{-\varepsilon}^{\varepsilon}e^{-\frac12|d_j|\cos(\phi)\zeta^2t}|\zeta|^{4-k}t^{1-\frac k2}d\zeta\\
		&\le C\sum_{j=1}^mt^{-\frac32}e^{-\frac{|x-c_jt|^2}{c'|d_j|t}}
\end{aligned}
\end{equation*}
for some $c,c'>0$ and $t\ge 1$. The estimate for $\mathcal F^{-1}(\chi_1I_2)$ is similarly.

Therefore, taking the $L^1$-norm in $x$ variable and using the Young inequality, we finish the proof.
\end{proof}

\begin{proposition}[$|x|>Ct$]\label{prop:multiplier 2 refined}
For $r\in[1,\infty]$, $m_j\in M_r$ with
\begin{equation}
	\|m^j\|_{M_r}\le Ct^{-1},\qquad j=1,2\textrm{ and }t\ge 1.
\end{equation}
\end{proposition}
\begin{proof}
The proof is the same as in Proposition \ref{prop:multiplier 2} and the fact that $e^{-\frac{|x|^2}{t}}$ is dominated by $t^{-\frac32}e^{-\frac{|x|^2}{2t}}$ since $|x|>Ct$ and $t\ge 1$. The proof is done.
\end{proof}
\section{Proof of main results}\label{sec:proofs of main theorems}
Recall the well-known inequality
\begin{lemma}[Interpolation inequality]
Let $(p_j,q_j)_{j\in\{0,1\}}$ be two elements
of $[1,\infty]^2$. Consider a linear operator $T$ which continuously
maps $L^{p_j}$ into $L^{q_j}$ for $j\in\{0,1\}$.
For any $\theta\in[0,1]$, if
\begin{equation*}
	\left(\dfrac{1}{p_{\theta}},\dfrac{1}{q_{\theta}}\right):=(1-\theta)\left(\dfrac{1}{p_0},
	\dfrac{1}{q_0}\right)+\theta\left(\dfrac{1}{p_1},\dfrac{1}{q_1}\right),
\end{equation*}
then $T$ continuous maps $L^{p_{\theta}}$ into $L^{q_{\theta}}$
and $\|T\|_{\mathcal{L}(L^{p_{\theta}};L^{q_{\theta}})}\le\|T\|_{\mathcal{L}(L^{p_0};L^{q_0})}^{1-\theta}\|T\|_{\mathcal{L}(L^{p_1};L^{q_1})}^{\theta}$.
\end{lemma}

We then introduce detailed proofs for Theorem \ref{theo:standard type} and Theorem \ref{theo:standard type symmetry}. 

\begin{proof}[Proof of Theorem \ref{theo:standard type}]
Let $u$ be the solution to \eqref{eq:original hyperbolic system}, recalling that $U=\sum_{j=1}^hU_j$ where $U_j$ is the solution to \eqref{eq:parabolic system} for $j\in\{1,\dots,h\}$ and $V=Q\sum_{j=1}^s V_j$ where $V_j$ is the solution to \eqref{eq:hyperbolic system} for $j\in\{1,\dots,s\}$. Then, we have
\begin{align*}
	u-U-V=\mathcal{F}^{-1}\bigl(\hat G-\hat K-\hat V\bigr)*u_{0}.
\end{align*}
On the other hand, let $\chi$ be the characteristic function, we have
\begin{equation*}
	\begin{aligned}
		\mathcal F^{-1}\bigl(\hat G-\hat K-\hat V\bigr)&=\mathcal F^{-1}\bigl[\bigl(\hat G-\hat K-\hat V\bigr)						\bigl(\chi_{[0,\varepsilon)}+\chi_{[\varepsilon,R]}+\chi_{(R,\infty)}\bigr)(|\xi|)\bigr]\\
		&=\mathcal F^{-1}\bigl[\bigl(\hat G-\hat K-\hat V\bigr)\bigl(\chi_{[0,\varepsilon)}+\chi_{[\varepsilon,R]}\bigr)(|\xi|)\bigr]\\
		&\hskip.25cm+\mathcal F^{-1}\bigl[\bigl(\hat G-\hat V\bigr)\chi_{(R,\infty)}(|\xi|)\bigr]-\mathcal F^{-1}\bigl[\hat K\chi_{(R,\infty)}(|\xi|)\bigr].
	\end{aligned}
\end{equation*}
Thus, since $\mathcal F^{-1}:L^1\to L^\infty$, we have
\begin{equation*}
	\begin{aligned}
		\bigl\|\mathcal F^{-1}\bigl(\hat G-\hat K-\hat V\bigr)\bigr\|_{L^{\infty}}
		&\le C\left[ \bigl\|\bigl(\hat G-\hat K-\hat V\bigr)\bigl(\chi_{[0,\varepsilon)}+\chi_{[\varepsilon,R]}\bigr)\bigr\|_{L^1}\right. \\
		&\left.+\bigl\|\hat K\chi_{(R,\infty)}\bigr\|_{L^1}\right]+\bigl\| \mathcal F^{-1}\bigl[\bigl(\hat G-\hat V\bigr)\chi_{(R,\infty)}(|\xi|)\bigr]\bigr\|_{L^{\infty}}.
	\end{aligned}
\end{equation*}
Hence, by the estimates \eqref{est:low}, \eqref{est:intermediate}, \eqref{est:high 1} and \eqref{est:high 2} in Proposition \ref{prop:fundamental solution standard type}, we obtain
\begin{equation*}
	\bigl\| u-U-V\bigr\|_{L^\infty}\le Ct^{-1}\|u_0\|_{L^1}, \qquad t\ge 1.
\end{equation*}
Furthermore, from Proposition \ref{prop:multiplier 1} and Proposition \ref{prop:multiplier 2}, for all $r\in[1,\infty]$, we also have
\begin{equation*}
	\bigl\| u-U-V\bigr\|_{L^r}\le Ct^{-\frac12}\|u_0\|_{L^r}, \qquad t\ge 1.
\end{equation*}
Therefore, by the interpolation inequality, we obtained the desired results.

The proof of \eqref{est:parabolic solution-hyperbolic solution} is similar and we finish the proof.
\end{proof}

\begin{proof}[Proof of Theorem \ref{theo:standard type symmetry}]
Similarly to before and from Propositions \ref{prop:low refined}, \ref{prop:multiplier 1 refined} and \ref{prop:multiplier 2 refined}. The proof is done.
\end{proof}

\section*{Appendix}

\subsection*{Eigenprojection computation}
In this subsection, we introduce a useful tool in order to compute the eigenprojection associated with a semi-simple eigenvalue of a matrix based on the determinant of this matrix and its minors. We start with some definitions

\noindent
-- a {\it set of indices} $\mathcal{I}$ is a set $\mathcal{I}=\{k_1<\dots < k_\ell\}$ with $k_p\in\{1,\dots,n\}$ for any $p$;\\
-- an {\it index-transformation} $\chi$ is an injective map from a set of indices $\mathcal{I}$ to $\{1,\dots,n\}$.\par
Then, we introduce some additional notations:\\
-- given two matrices $A, B\in  M_n(\mathbb R)$, a set of indices $\mathcal{I}=\{i\}$ and
an index-transformation $\chi$, we denote by $\Phi(A,B;\mathcal{I},\chi)$ the matrix obtained by
substituting the $i$-th column of $A$ by the $\chi(i)$-th column of $B$.\\
-- given sets of indices $\mathcal I,\mathcal K,\mathcal L$ satisfying $\mathcal K\subseteq \mathcal I$ and $|\mathcal K|=|\mathcal L|$, where $|\mathcal K|$ and $|\mathcal L|$ are the cardinalities of $\mathcal K$ and $\mathcal L$ respectively. A map $\chi_{\mathcal K\to \mathcal L}$ is an injective map from $\mathcal I$ to $\{1,\dots,n\}$ defined by $\chi_{\mathcal K\to \mathcal L}(k)=k$ if $k\not\in \mathcal K$; and there is a unique $\ell \in\mathcal L$ such that $\chi_{\mathcal K\to \mathcal L}(k)=\ell$ for each $k\in\mathcal K$.

We then set $[A]^k$ the $n\times n$ matrix with components defined by
\begin{equation*}
	[A]^k_{ij}:=\sum \det\Phi( A,I; \mathcal{I},\chi_{i\to j}),
	\qquad i,j\in\{1,\dots,n\},
\end{equation*}
where the sum is made on sets of indices $\mathcal {I}$ containing $i$ and with cardinality $|\mathcal{I}|=k+1$. If $k=0$, then $[A]^0=\textrm{\normalfont adj}\,(A)$. Hence, the above notation can be seen as an extended version of the adjunct of the matrix $A$.

One sets
\begin{align}
\label{func:P}	\mathbb{P}_k(A)&:=\dfrac{(k+1)[A]^{k}}{\textrm{\normalfont Tr}\, [A]^{k}},\\
\label{func:S}	\mathbb{S}_k(A)&:=\dfrac{(k+1)(k+2)[A]^{k+1}\textrm{\normalfont Tr}\,[A]^{k}-(k+1)^2[A]^{k}\textrm{\normalfont Tr}\, [A]^{k+1}}{(k+2)\bigl(\textrm{\normalfont Tr}\,[A]^k)^2}.
\end{align}

Let $\Gamma$ be an oriented closed curve enclosing $0$ except for the other eigenvalues of $A$ in the resolvent set $\rho(A)$, one defines
\begin{equation}\label{eq:eigenprojections}
	P:=-\dfrac{1}{2\pi i}\int_{\Gamma}(A-zI)^{-1}\,dz\qquad \textrm{and}\qquad S:=\dfrac{1}{2\pi i}\int_{\Gamma}z^{-1}(A-zI)^{-1}\,dz.
\end{equation}
The matrix $P\in M_n(\mathbb R)$ is called the eigenprojection associated with $0$ and the matrix $S\in M_n(\mathbb R)$ is called the reduced resolvent coefficient associated with $0$.
\begin{proposition}\label{prop:eigenprojection computation}
Let $A\in M_n(\mathbb R)$ and $0 \in\sigma(A)$ is semi-simple with algebraic multiplicity $m\ge 1$, then the eigenprojection associated with $a$ has the formula
\begin{equation}\label{eq:eigenprojection formula}
	P=\mathbb P_{m-1}(A)\qquad \textrm{and}\qquad S=\mathbb S_{m-1}(A).
\end{equation}
\end{proposition}

Before going to the proof of Proposition \ref{prop:eigenprojection computation}, we introduce the following results.
\begin{lemma}\label{mainlemma} Let $f(x):=\det\Phi(A+xB,C;\mathcal{I},\chi_{\mathcal{K}\to \mathcal{L}} )$ for $x\in\mathbb{C}$. By setting
\begin{equation*}
	\mathcal{J}:=\bigl\{h \notin \mathcal{I}:\exists j \in \chi_{\mathcal K\to \mathcal L}(\mathcal I), c_j
= b_h\in \textrm{\normalfont col}(B) \bigr\}
\end{equation*}
where $\textrm{\normalfont col}(B)$ is the column space of $B$, for any non negative integer $m$ satisfying $m\le n-|\mathcal I|$, the following holds
\begin{equation*}\label{dfm}
	f^{(m)}(x)=m!\sum\det \Phi(\Phi(A+xB,C;\mathcal{I},
	\chi_{\mathcal{K}\to \mathcal{L}}),B;\mathcal{M},\chi_{\mathcal{M}\to \mathcal M})
\end{equation*}
where the sum is made on the set of indices $\mathcal {M}$ with cardinality $|\mathcal{M}|=m$
and $\mathcal{M}\cap( \mathcal{I}\cup\mathcal{J})=\emptyset$.

\end{lemma}
\begin{proof} By definition, we have
\begin{equation*}
	f(x)=\bigwedge_{h=1}^{n}f_h(x)\qquad \textrm{where}\qquad f_h(x):=\begin{cases}(a_h+b_hx) & \textrm{if } h\notin \mathcal I,\\ c_{\chi_{\mathcal K\to \mathcal L}(h)}&\textrm{if } h\in \mathcal I.
	\end{cases}
\end{equation*}
Hence, the derivative of order $m\in\mathbb N$ of $f$ satisfies
\begin{equation*}\label{derivative of order m}
	f^{(m)}(x)=\sum \bigwedge_{h=1}^{n}f_h^{(s_h)}(x)\,\, \textrm{where }\,f_h^{(s_h)}(x):=\begin{cases}(a_h+b_h x)^{(s_h)} & \textrm{if } h \notin \mathcal I,\\ c_{\chi_{\mathcal K\to \mathcal L}(h)}^{(s_h )} &\textrm{if } h\in \mathcal I
	\end{cases}
\end{equation*}
where $s_h:=s_h^1+\dots+s_h^m\in\mathbb N$ with $s_h^\ell \in\{0,1\}$ for all $\ell\in\{1,\dots,m\}$ and if we denote by $S\in\{1,0\}^{n\times m}$ the matrix defined by $S_{h\ell}:=s_h^\ell$, the sum is made on the set 
\begin{equation*}
	\mathcal S:=\left\{ S : s_1+\dots+s_n=m\right\}.
\end{equation*}
It means that $f^{(m)}(x)$ is the sum of a finite number of determinants, where the determinants are generated by the elements $S$ of $\mathcal S$ and they are given by $D_S:=\bigwedge_{h=1}^{n}f_h^{(s_h)}(x)$.

Moreover, for any matrix $S\in\mathcal S$, if $s_h\ge 2$ for
$h\notin \mathcal I$, then $(a_h+b_hx)^{(s_h)}=O_{n\times 1}$ and if $s_h \ge 1$ for $h \in\mathcal I$, $c_{\chi_{\mathcal K\to \mathcal L}(h)}^{(s_h)}=O_{n\times 1}$; and thus, the determinants related to these cases are zero. Thus, due to the condition $s_1+\dots+s_n=m$ where $m\le n-|\mathcal I|$, we can introduce a partition for $\mathcal S$ such that its elements denoted by $\mathcal S_{\mathcal M}$ are associated with index-sets $\mathcal M:=\{h_1,\dots,h_m\}\subset \{1,\dots,n\}\backslash \mathcal I$ and they are given by
\begin{equation*}
	\mathcal S_{\mathcal M}:=\left\{S: s_h=\delta_{\mathcal M}(h) \textrm{ for all } h=1,\dots,n\right\}
\end{equation*}
where
\begin{equation*}\label{condition K}
	\delta_{\mathcal M}(h):=\left\{\begin{matrix} 1&\textrm{ if } h \in\mathcal M,\\
	 0& \textrm{ if }h \notin\mathcal M.\end{matrix}\right.
\end{equation*}
In particular, for any $\mathcal M$, if $S$ and $S'$ belong to $\mathcal S_{\mathcal M}$, one has $D_S=D_{S'}$ since $s_h=\delta_{\mathcal M}(h)=s'_h$ for all $h \in\{1,\dots,n\}$ where $s_h$ and $s'_h$ are the sum of the elements of the $h$-th rows of the matrices $S$ and $S'$ respectively. On the other hand, we have
\begin{equation*}
\begin{aligned}
	D_S&=\bigwedge_{h=1}^{n}f_h^{(s_h)}(x)\qquad\textrm{where}\qquad f_h^{(s_h)}(x):=\begin{cases}b_h & \textrm{if } h \in \mathcal M,\\
	(a_h+b_h x) &\textrm{if } h\notin \mathcal M\cup \mathcal I, \\ c_{\chi_{\mathcal K\to \mathcal L}(h)} &\textrm{if } h\in \mathcal I
	\end{cases}\\
	&= \det \Phi\bigl(\Phi\bigl(A+xB,C;\mathcal I,\chi_{\mathcal K\to \mathcal L}\bigr),B;\mathcal M,\chi_{\mathcal M\to \mathcal M}\bigr).
\end{aligned}
\end{equation*}
Moreover, let $\sigma$ be a permutation of the set $\{e_1,\dots,e_m\}$ where $e_\ell$ the $\ell$-th row of the identity matrix $I_{m\times m}$. Then, by definition, the rows of any matrix $S\in\mathcal S_{\mathcal M}$ for $\mathcal M=\{h_1,\dots,h_m\}$ must be in the form of
\begin{equation*}
	\begin{pmatrix} S_{h_\ell 1} &\dots&S_{h_\ell m} \end{pmatrix}=\begin{cases} \sigma(e_\ell) &\textrm{ if } \ell \in\{1,\dots,m\},\\
	O_{1\times m} &\textrm{ if } \ell\in\{m+1,\dots,n\}.\end{cases}
\end{equation*}
Therefore, since the $D_S=D_{S'}$ for $S,S'\in\mathcal S_{\mathcal M}$ and since the number of $\sigma$ is $m!$, we obtain
\begin{equation*}
\begin{aligned}
	f^{(m)}(x)&=\sum_{\mathcal M}\sum_{S\in\mathcal S_{\mathcal M}}D_S\\
	&=m!\sum_{\mathcal M}\det \Phi\bigl(\Phi\bigl(A+xB,C;\mathcal I,\chi_{\mathcal K\to \mathcal L}\bigr),B;\mathcal M,\chi_{\mathcal M\to \mathcal M}\bigr).
\end{aligned}
\end{equation*}

Assuming that there is $\mathcal M$ to satisfy $\mathcal M\cap \mathcal J\ne \emptyset$, then there exists $h \in\mathcal M$ such that $h \notin \mathcal I$ and $b_h=c_j$ for some $j\in\chi_{\mathcal K\to \mathcal L}(\mathcal I)$. Hence, the determinants $D_S$ generated by $S\in\mathcal S_{\mathcal M}$ have
\begin{equation*}
	f_h^{(s_h)}(x)=b_h=c_j.
\end{equation*}
Furthermore, since $j\in\chi_{\mathcal K\to \mathcal L}(\mathcal I)$, there is $i\in\mathcal I$ such that $\chi_{\mathcal I\to \mathcal J}(i)=j$. Thus, the determinants $D_S$ also have
\begin{equation*}
	f_i^{(s_i)}(x)=c_{\chi_{\mathcal K\to \mathcal L}(i)}=c_j
\end{equation*}
as well. Since $h \notin\mathcal I$ and $i\in\mathcal I$, it implies that there are at least two columns are the same and the determinants $D_S$ are accordingly zero. Therefore, we can choose $\mathcal M\cap \mathcal J=\emptyset$ for all $\mathcal M$. The proof is done.
\end{proof}

\begin{corollary}\label{adjoint}
Let $h\in\{0,\dots,n-1\}$ and $A\in M_n(\mathbb R)$, for $x\in \mathbb C$, the following hold
\begin{equation*}
	[A]^h=(h!)^{-1}(\textrm{\normalfont adj}(A+xI))^{(h)}\big|_{x=0}, \qquad\textrm{\normalfont Tr}\,[A]^h=(h!)^{-1}(\det(A+xI))^{(h+1)}\big|_{x=0} 
\end{equation*}
where $\textrm{\normalfont adj}(A+xI)$ is the adjoint matrix of $A+xI$.
\end{corollary}
\begin{proof}
By the definition of the adjoint matrix, the elements of the matrix are the minors of the matrix $A+xI$. On the other hand, one defines
\begin{equation*}
	M_{ji}(x):=\det \Phi(A+xI,I;\{i\},\chi_{i\to j}), \qquad i,j\in\{1,\dots,n\}.
\end{equation*}
Thus, by the definition of the minor and the definition of $\Phi$, we have $(\textrm{\normalfont adj}(A+xI))_{ij}=M_{ji}(x)$ for all $i,j\in\{1,\dots,n\}$. Moreover, by Lemma \ref{mainlemma}, the following holds
\begin{equation*}
\begin{aligned}
	M_{ji}^{(h)}(0)&=h!\sum_{\mathcal{H}\not\ni i,j}\det\Phi(\Phi(A,I;\{i\},\chi_{i\to j}),I;\mathcal H,\chi_{\mathcal H\to \mathcal H})\\
	&=h!\sum_{\mathcal H\cup \{i\}}\det\Phi(A,I;\mathcal H \cup \{i\},\chi_{i\to j})=h![A]^h_{ij}
\end{aligned}
\end{equation*}
where $\mathcal H$ has the cardinality $|\mathcal H|=h$ for $h\in\{0,\dots,n-1\}$. We finished proving the first equality in the statement.

By the definition of $\chi_{\mathcal M \to \mathcal N}:\mathcal I\to \{1,\dots,n\}$ for any set of indices $\mathcal I$ containing $\mathcal M$, it follows that $\chi_{i\to i}\equiv \chi_{\mathcal I\to \mathcal I}$ for any $\mathcal I$ containing $i$. Thus, by the definition of $[A]^h$ for $h\in\{0,\dots,n-1\}$, we have
\begin{equation*}
\begin{aligned}
	\textrm{\normalfont Tr}\,[A]^{h}&=\sum_{i=1}^n\sum_{\mathcal I\ni i}\det \Phi(A,I;\mathcal I,\chi_{i\to i})\\
	&=\sum_{\mathcal I\ni 1}\det \Phi(A,I;\mathcal I,\chi_{\mathcal I \to \mathcal I})+\dots+\sum_{\mathcal I\ni n}\det \Phi(A,I;\mathcal I,\chi_{\mathcal I\to \mathcal I})
\end{aligned}
\end{equation*}
where $\mathcal I$ has the cardinality $|\mathcal I|=h+1$. 

Moreover, for any fixed set of indices $\mathcal I$ satisfying $|\mathcal I|=h+1$, $\mathcal I$ must be considered in $h+1$ terms in the right hand side of the formula of $\textrm{\normalfont Tr}\,[A]^{h}$. In fact, each of $i\in\mathcal I$ belongs to $\{1,\dots,n\}$. Thus, for any fixed $\mathcal I$, we can collect $h+1$ quantities that are the same and we have
\begin{equation*}
	\textrm{\normalfont Tr}\,[A]^{h}=(h+1)\sum_{\mathcal I}\det \Phi(A,I;\mathcal I,\chi_{\mathcal I \to \mathcal I})
\end{equation*}
where $\mathcal I$ has the cardinality $|\mathcal I|=h+1$.

Furthermore, by Lemma \ref{mainlemma}, one has
\begin{equation*}
	(\det(A+xI))^{(h+1)}\big|_{x=0}=(h+1)!\sum_{\mathcal I}\det \Phi(A,I;\mathcal I,\chi_{\mathcal I \to \mathcal I}).
\end{equation*}
The proof is done.
\end{proof}

We can now give a proof for Proposition \ref{prop:eigenprojection computation}.
\begin{proof}[Proof of Proposition \ref{prop:eigenprojection computation}]
By definition, the resolvent of the matrix $A$ is given by
\begin{equation*}
	R(z):=(A-zI)^{-1}=\dfrac{\textrm{\normalfont adj}(A-zI)}{\det(A-zI)}.
\end{equation*}
For $z$ small, the resolvent can be expanded as
\begin{equation*}
	R(z)=\dfrac{1}{z^m}\dfrac{\sum_{h=0}^{n-1}(-1)^h(h!)^{-1}(\textrm{\normalfont adj}(A+xI))^{(h)}\big|_{x=0}z^h}{\sum_{h=m}^n(-1)^h(h!)^{-1}(\det(A+xI))^{(h)}\big|_{x=0}z^{h-m}}.
\end{equation*}
Thus, Corollary \ref{adjoint} implies that
\begin{equation*}
	R(z)=\dfrac{1}{z^m}\dfrac{\sum_{h=0}^{n-1}(-1)^h[A]^hz^h}{\sum_{h=m}^n(-1)^hh^{-1}\bigl(\textrm{\normalfont Tr}\,[A]^{h-1}\bigr)z^{h-m}}.
\end{equation*}
On the other hand, by using the Laurent expansion of $R(z)$ (see \citep{kato}), we also have
\begin{equation*}
	R(z)=-\sum_{h=-1}^{+\infty}z^{-h-1}(N)^h-z^{-1}P+\sum_{h=0}^{+\infty}z^h(S)^{h+1},
\end{equation*}
where $P,S$ are in \eqref{eq:eigenprojections} and $N=AP$ is the nilpotent matrix associated with the eigenvalue $0$ of $A$.
Then equating two sides, we obtain the formulas. We finish the proof.
\end{proof}

\subsection*{Perturbation theory for linear operators}
In this subsection, we introduce some results from the perturbation theory for linear operators in finite dimensional space that we will use for this paper. Moreover, we will sketch the proofs of them. For whom is interested in, see \citep{kato} for more details.

\begin{proposition}\label{prop:subprojections}
	Assume that $T$ is a matrix operator considered in a domain $\mathcal D:=\textrm{\normalfont ran}(P)$ where $P$ is a matrix operator. Let $(P_j)$ for $j=1,\dots,k$ be a sequence of matrix operators such that
\begin{equation}\label{eq:subprojections}
	P_j^2=P_j,\quad P_jP_{j'}=O \textrm{ for } j\ne j', \quad P=\sum_{j=1}^kP_j \quad \textrm{and}\quad \textrm{\normalfont ran}(P)=\bigoplus_{j=1}^k\textrm{\normalfont ran}(P_j).
\end{equation}
If $T$ commutes with $P_j$ for $j=1,\dots,k$, then one has
\begin{equation}\label{eq:commuting with operator}
	TP_j=P_jT=P_jTP_j\qquad\textrm{and}\qquad TP=PT=PTP=\sum_{j=1}^k(TP_j).
\end{equation}

Moreover, $\lambda\in\sigma(T)$ considered in $\textrm{\normalfont ran}(P)$ if and only if there is $j_0\in\{1,\dots,k\}$ such that $\lambda\in\sigma(T)$ considered in $\textrm{\normalfont ran}(P_{j_0})$.
\end{proposition}
\begin{proof}
For $j=1,\dots,k$, since $P_j^2=P_j$ by \eqref{eq:subprojections}, one has $P_jTP_j=TP_j^2=TP_j=P_jT$ if $T$ commutes with $P_j$.

Also from \eqref{eq:subprojections}, one has $P=\sum_{j=1}^kP_j$ and thus, we have
\begin{equation*}
	TP=\sum_{j=1}^k(TP_j)=\sum_{j=1}^kP_jT=PT.
\end{equation*}
We now prove that $P$ is a projection. Indeed, since $P_jP_{j'}=O$ for $j\ne j'$ and $P_j^2=P_j$, we have
\begin{equation*}
	P^2=\left(\sum_{j=1}^kP_j\right)^2=\sum_{j,j'=1}^kP_jP_{j'}=\sum_{j=1}^kP_j=P.
\end{equation*}
Hence, we have $PTP=P^2T=PT=TP=\sum_{j=1}^k(TP_j)$.

Assume that there is $u\in\textrm{\normalfont ran}(P)$ such that $u\ne O_{n\times 1}$ and $Tu=\lambda u$. Then, $u=Pu$ and one has $TPu=\lambda Pu$. Moreover, since $PP_j=\sum_{j'=1}^k(P_{j'}P_j)=P_j=\sum_{j'=1}(P_jP_{j'})=P_jP$ and $TP_j=P_jT$, one obtains
\begin{equation*}
	TP_ju=T(P_jP)u=(P_jT)Pu=\lambda P_jPu=\lambda P_ju.
\end{equation*}
On the other hand, since the direct sum $\sum_{j=1}^k(P_ju)=Pu=u\ne O_{n\times 1}$, there is at least $j_0\in\{1,\dots,k\}$ such that $P_{j_0}u\ne O_{n\times 1}$. Thus, let $v=P_{j_0}u\in\textrm{\normalfont ran}(P_{j_0})$, $v\ne O_{n\times 1}$ and $Tv=\lambda v$.

For the inverse, let $v\in\textrm{\normalfont ran}(P_{j_0})$ for some $j_0\in\{1,\dots,k\}$ such that $v\ne O_{n\times 1}$ and $Tv=\lambda v$, since $\textrm{\normalfont ran}(P_{j_0})\subset \textrm{\normalfont ran}(P)$ by \eqref{eq:subprojections}, we finish the proof.
\end{proof}

\begin{proposition}\label{prop:construction of subprojections}
	For $x\in\mathbb C$ small enough, let $T(x)=T^{(0)}+\mathcal O(|x|)$ where $T^{(0)}$ is a matrix and $T$ is considered in the domain $\mathcal D:=\textrm{\normalfont ran}(P)$ where $P(x)=P^{(0)}+\mathcal O(|x|)$. Assume that there are $k\le n$ distinct eigenvalues $\lambda_j^{(0)}$ of $T^{(0)}$ considered in $\textrm{\normalfont ran}(P^{(0)})$ where $j=1,\dots,k$. Then, there is a unique sequence $(P_j)$ satisfying \eqref{eq:subprojections} and \eqref{eq:commuting with operator} such that $P_j(x)=P_j^{(0)}+\mathcal O(|x|)$ where $P_j^{(0)}$ is the eigenprojection associated with $\lambda_j^{(0)}$ where $j=1,\dots,k$.
	
	In particular, for any $\lambda\in\sigma(T)$ considered in $\textrm{\normalfont ran}(P)$, $\lambda\in\sigma(T)$ considered in $\textrm{\normalfont ran}(P_j)$ if and only if $\lambda(x)\to \lambda_j^{(0)}$ as $|x|\to 0$ for $j\in\{1,\dots,k\}$.
\end{proposition}

Before going to the proof of Proposition \ref{prop:construction of subprojections}, we have the following lemma. Let the resolvent of a matrix operator $T$ where $T$ depends on $x\in\mathbb C$ be
\begin{equation}
	R(x,z):=(T(x)-zI)^{-1},\qquad z\in \rho(T).
\end{equation}
\begin{lemma}\label{lem:resolvent}
The resolvent of the matrix operator $T(x):=T^{(0)}+\mathcal O(|x|)$ is holomorphic in any neighborhood of $(x,y)\in\mathbb C^2$ such that $y\in \rho\bigl(T^{(0)}\bigr)$. Moreover, if $\Gamma$ a compact subset of $\rho(T^{(0)})$, then $R(x,y)$ is a convergent series as $|x|\to 0$ uniformly in $y\in\Gamma$ and thus one has the expansion
\begin{equation}\label{eq:expansion of resolvent}
	R(x,y)=R^{(0)}(y)+\mathcal O(|x|), \qquad |x|\to 0,
\end{equation}
where $R^{(0)}(y):=(T^{(0)}-yI)^{-1}$.

As a consequence, there is no eigenvalue of $T$ included in $\Gamma$.
\end{lemma} 
\begin{proof}[Proof of Lemma \ref{lem:resolvent}]
For $z\in\rho(T)$ and $y\in \rho\bigl(T^{(0)}\bigr)$, we have
\begin{equation*}
\begin{aligned}
	T(x)-zI&=\bigl(T^{(0)}-yI\bigr)-\left((z-y)I-\bigl(T(x)-T^{(0)}\bigr)\right)\\
	&=\left(1-\left((z-y)I-\bigl(T(x)-T^{(0)}\bigr)\right)\bigl(T^{(0)}-\lambda^{(0)}I\bigr)^{-1}\right)\bigl(T^{(0)}-yI\bigr).
\end{aligned}
\end{equation*}
Thus, taking the inverse and since $T(x)-T^{(0)}=\mathcal O(|x|)$ for $x$ small, we obtain
\begin{equation*}
	R(x,z)=R^{(0)}(y)\left(1-\left((z-y)I-\mathcal O(|x|)\right)R^{(0)}(y)\right)^{-1}.
\end{equation*}
Furthermore, for any matrix norm $\|\cdot\|$, we also have
\begin{equation*}
\left\|\left((z-y)I-\mathcal O(|x|)\right)R^{(0)}(y)\right\|\le \left(|z-y|+C|x|\right)\bigl\|R^{(0)}(y)\bigr\|<1
\end{equation*}
for $x$ and $z-y$ small enough. Thus, it implies that $R(x,z)$ can be expanded as a convergent series and is holomorphic in any neighborhood of $(x,y)$.

On the other hand, for $x$ small and $y\in\rho\bigl(T^{(0)}\bigr)$, one has
\begin{equation*}
	T(x)-yI=\bigl(T^{(0)}-yI\bigr)+\mathcal O(|x|)=\bigl(I+\mathcal O(|x|)\bigl(T^{(0)}-yI\bigr)^{-1}\bigr)\bigl(T^{(0)}-yI\bigr).
\end{equation*}
Thus, one deduces
\begin{equation*}
	R(x,y)=R^{(0)}(y)\bigl(I+\mathcal O(|x|)R^{(0)}(y)\bigr)^{-1}=R^{(0)}(y)\bigl(I+\mathcal O(|x|)\bigr)=R^{(0)}(y)+\mathcal O(|x|).
\end{equation*}
On the other hand, one notes that $R(x,y)$ is expressed based on $R^{(0)}(y)$. Since $\Gamma$ is a compact subset of $\rho\bigl(T^{(0)}\bigr)$, the norm $\big\|\mathcal O(|x|)R^{(0)}(y)\bigr\|$ can be bounded by $1$ uniformly for all $y\in\Gamma$. As a consequence, since $R(x,y)$ exists for all $x$ small and $y\in \Gamma$, there is no eigenvalue of $T$ belongs to $\Gamma$.
\end{proof}

We are now going back to the proof of Proposition \ref{prop:construction of subprojections}.
\begin{proof}[Proof of Proposition \ref{prop:construction of subprojections}]
Primarily, we have the follows. Let $\lambda\in \sigma(T)$ considered in $\mathbb C^n$, $\lambda$ must be a solution of the dispersion polynomial $p:=\det(T-\lambda I)$ that is an analytic function in $x\in \mathbb C$ since $T$ is analytic in $x\in \mathbb C$. Moreover, it is known that $\lambda$ is continuous and converges to an eigenvalue of $T^{(0)}$ as $|x|\to 0$ since $T(x)=T^{(0)}+\mathcal O(|x|)$ as $|x|\to 0$. Thus, one can write
\begin{equation}\label{eq:formula of eigenvalue of T}
	\lambda(x):=\lambda^{(0)}+{\scriptstyle\mathcal O}(1),\qquad |x|\to 0,
\end{equation}
where $\lambda^{(0)}\in\sigma\bigl(T^{(0)}\bigr)$ considered in $\mathbb C^n$ is the limit of $\lambda$ as $|x|\to 0$. In particular, due to the formula \eqref{eq:formula of eigenvalue of T}, the eigenvectors $u\in\mathbb C^{n}$ associated with $\lambda$ can be chosen such that
\begin{equation}\label{eq:formula of eigenvector of lambda}
	u(x):=u^{(0)}+{\scriptstyle\mathcal O}(1),\qquad |x|\to 0,
\end{equation}
where $u^{(0)}\in\mathbb C^n$ are the eigenvectors associated with $\lambda^{(0)}$. It follows that $u\in \textrm{\normalfont ran}(P)$ if and only if $u^{(0)}\in\textrm{\normalfont ran}\bigl(P^{(0)}\bigr)$. Indeed, one has
\begin{equation*}
	Pu=\bigl(P^{(0)}+\mathcal O(|x|)\bigr)\bigl(u^{(0)}+{\scriptstyle\mathcal O}(1)\bigr)=P^{(0)}u^{(0)}+{\scriptstyle\mathcal O}(1)
\end{equation*}
and thus $Pu=u$ if and only if $P^{(0)}u^{(0)}=u^{(0)}$. It implies that $\lambda\in\sigma(T)$ considered in $\textrm{\normalfont ran}(P)$ if and only if $\lambda^{(0)}\in\sigma\bigl(T^{(0)}\bigr)$ considered in $\textrm{\normalfont ran}\bigl(P^{(0)}\bigr)$. Therefore, if $\lambda_j^{(0)}$ for $j=1,\dots,k$ are the $k$ distinct eigenvalues of $T^{(0)}$ considered in $\textrm{\normalfont ran}\bigl(P^{(0)}\bigr)$, then the above argument and the expansion \eqref{eq:formula of eigenvalue of T} show that for any eigenvalue $\lambda$ of $T$ considered in the domain $\mathcal D=\textrm{\normalfont ran}(P)$, then $\lambda$ converges to an eigenvalue $\lambda_j^{(0)}$ of $T^{(0)}$ considered in $\textrm{\normalfont ran}\bigl(P^{(0)}\bigr)$ for some $j\in\{1,\dots,k\}$ as $|x|\to 0$. In particular, for each $j\in\{1,\dots,k\}$, the set of all eigenvalues $\lambda$ of $T$ considered in $\mathcal D$ such that $\lambda\to \lambda_j^{(0)}$ as $|x|\to 0$ is the $\lambda_j^{(0)}$-group of $T$. For easy, we consider the formal formula
\begin{equation}\label{eq:subgroups 1}
	\sigma(T)\textrm{ considered in } \mathcal D =\bigcup_{j=1}^k\bigl(\lambda_j^{(0)}\textrm{-group}\bigr),
\end{equation}
where
\begin{equation}\label{eq:subgroups 2}
	\lambda_j^{(0)}\textrm{-group}:=\bigl\{\lambda\in \sigma(T)\textrm{ considered in } \mathcal D :\lambda\to \lambda_j^{(0)} \textrm{ as }|x|\to 0\bigr\}.
\end{equation}

We are going to prove the unique existence of a sequence $(P_j)$ satisfying \eqref{eq:subprojections} and \eqref{eq:commuting with operator} where $j=1,\dots,k$. First of all, we consider the domain $\mathcal D=\mathbb C^n$ {\it i.e.} $P=I$ the identity matrix and since $P(x)=P^{(0)}+\mathcal O(|x|)$ as $|x|\to 0$, $P^{(0)}=I$ as well. Hence, the eigenvalues $\lambda$ of $T$ and $\lambda_{j}^{(0)}$ of $T^{(0)}$ in this case are considered in $\mathbb C^n$. Let $\lambda\in\sigma(T)$ and let $\Gamma_\lambda$ be a closed curve enclosing $\lambda$ except for the other eigenvalues of $T$ in the complex plane, since $\lambda$ is singularity of the resolvent $R(z)=(T-zI)^{-1}$ of $T$, the Cauchy integral
\begin{equation}
	P_\lambda(x):=-\dfrac{1}{2\pi i}\int_{\Gamma_\lambda}R(x,z)\,dz
\end{equation}
is exactly the eigenprojection associated with $\lambda$. The matrix operator $N_\lambda:=(T-\lambda I)P_\lambda$
is then the nilpotent part associated with $\lambda$. Moreover, $TP_\lambda=\lambda P_\lambda+N_\lambda=P_\lambda T$. Nonetheless, the resolvent $R(x,z)$ for $x$ small and $z\in\rho(T)$ cannot be expanded explicitly in general except for the case $z$ belongs to a compact set contained in $\rho\bigl(T^{(0)}\bigr)$ provided Lemma \ref{lem:resolvent}.

Based on that, for $j\in\{1,\dots,k\}$, let $\Gamma_j$ be a closed curve in $\rho\bigl(T^{(0)}\bigr)$ such that $\Gamma_j$ encloses the eigenvalue $\lambda_j^{(0)}$ except for the other eigenvalues of $T^{(0)}$. Then, by Lemma \ref{lem:resolvent}, there is no eigenvalue of $T$ to belong to $\Gamma_j$ and therefore, for $x$ small enough, the interior domain bounded by $\Gamma_j$ only encloses the eigenvalues of $T$ such that $\lambda\to \lambda_j^{(0)}$ as $|x|\to 0$ {\it i.e.} the $\lambda_j^{(0)}$-group is contained in this domain except for the other groups of $T$. Hence, for every $\lambda$ included in the $\lambda_j^{(0)}$-group, one can choose $\Gamma_\lambda$ such that they are strictly contained in the domain bounded by $\Gamma_j$ and one has
\begin{equation*}
	P_j(x):=-\dfrac{1}{2\pi i}\int_{\Gamma_j}R(x,y)\,dy=\sum_{\lambda\in \lambda_j^{(0)}\textrm{-group}}\left(-\dfrac{1}{2\pi i}\int_{\Gamma_\lambda}R(x,z)\,dz\right)=\sum_{\lambda\in \lambda_j^{(0)}\textrm{-group}}P_\lambda.
\end{equation*}
Hence, $P_j$ is called the {\it total projection} associated with the $\lambda_j^{(0)}$-group of $T$.

The sequence of the total projections $P_j$ where $j\in\{1,\dots,k\}$ satisfies the properties \eqref{eq:subprojections}. Indeed, since, for $\lambda\in \lambda_j^{(0)}$-group, $P_\lambda$ is an eigenprojection, one has
\begin{equation*}
	P_j^2=\left(\sum_{\lambda\in \lambda_j^{(0)}\textrm{-group}}P_\lambda\right)^2=\sum_{\lambda,\lambda'\in \lambda_j^{(0)}\textrm{-group}}P_\lambda P_{\lambda'}=\sum_{\lambda\in \lambda_j^{(0)}\textrm{-group}}P_\lambda=P_j,
\end{equation*}
and for $j\ne j'$, since $\lambda_j^{(0)}\ne \lambda_{j'}^0$ due to the distinct property, one has
\begin{equation*}
	P_jP_{j'}=\left(\sum_{\lambda\in \lambda_j^{(0)}\textrm{-group}}P_\lambda\right)\left(\sum_{\lambda'\in \lambda_{j'}^{(0)}\textrm{-group}}P_{\lambda'}\right)=\sum_{\substack{ \lambda\in \lambda_j^{(0)}\textrm{-group} \\  \lambda'\in \lambda_{j'}^{(0)}\textrm{-group} }}P_\lambda P_{\lambda'}=O
\end{equation*}
since these two groups are distinct. Moreover, we have $\mathbb C^n=\bigoplus_{\lambda\in \sigma(T)}\textrm{\normalfont ran}(P_\lambda)$ and $I=\sum_{\lambda \in \sigma(T)}P_\lambda$ and thus from \eqref{eq:subgroups 1} and \eqref{eq:subgroups 2}, one deduces
\begin{equation*}
	I=\sum_{j=1}^k\sum_{\lambda\in \lambda_j^{(0)}\textrm{-group}}P_\lambda=\sum_{j=1}^kP_j \quad \textrm{and}\quad \mathbb C^n=\bigoplus_{j=1}^k\bigoplus_{\lambda\in \lambda_j^{(0)}\textrm{-group}}\textrm{\normalfont ran}(P_\lambda)=\bigoplus_{j=1}^k\textrm{\normalfont ran}(P_j).
\end{equation*}

Then, the property \eqref{eq:commuting with operator} holds if one proves that $T$ commutes with $P_j$ for all $j\in\{1,\dots,k\}$ due to Proposition \ref{prop:subprojections}. Infact, for all $\lambda\in \sigma(T)$, since $TP_\lambda=P_\lambda T$, one obtain that
\begin{equation*}
	TP_j=\sum_{\lambda\in \lambda_j^{(0)}\textrm{-group}}(TP_\lambda)=\sum_{\lambda\in \lambda_j^{(0)}\textrm{-group}}(P_\lambda T)=P_jT
\end{equation*}
for all $j\in\{1,\dots,k\}$.

On the other hand, from \eqref{eq:expansion of resolvent}, one has
\begin{equation*}
	P_j(x):=-\dfrac{1}{2\pi i}\int_{\Gamma_j}R(x,y)\,dy=-\dfrac{1}{2\pi i}\int_{\Gamma_j}R^{(0)}(y)\,dy+\mathcal O(|x|),\qquad |x|\to 0,
\end{equation*}
where $R^{(0)}(y)=\bigl(T^{(0)}-yI\bigr)^{-1}$ for $y\in\rho\bigl(T^{(0)}\bigr)$. Then, it is easy to see that $R^{(0)}$ is the resolvent of the matrix $T^{(0)}$ and thus by the definition of $\Gamma_j$, it implies that 
\begin{equation*}
	P_j(x)=P_j^{(0)}+\mathcal O(|x|),\qquad |x|\to 0,
\end{equation*}
where $P_j^{(0)}$ is the eigenprojection associated with $\lambda_j^{(0)}$ since it is known that $P_j^{(0)}=-\dfrac{1}{2\pi i}\int_{\Gamma_j}R^{(0)}(y)\,dy$.

We already construct the desired sequence of $P_j$ where $j=1,\dots,k$ if $\mathcal D=\mathbb C^n$. For the case $\mathcal D=\textrm{\normalfont ran}(P)$, it is enough to define the unique eigenprojection $\tilde P_j$ associated with the domain $\textrm{\normalfont ran}(P_j)\cap \textrm{\normalfont ran}(P)$ where $P_j$ is constructed as before for each $j\in\{1,\dots,k\}$. One can denote $\tilde P_j$ again by $P_j$ where $j=1,\dots,k$.

Finally, we prove that for any $\lambda\in\sigma(T)$ considered in $\textrm{\normalfont ran}(P)$, $\lambda\in\sigma(T)$ considered in $\textrm{\normalfont ran}(P_j)$ if and only if $\lambda(x)\to \lambda_j^{(0)}$ as $|x|\to 0$ for $j\in\{1,\dots,k\}$. Indeed, for each $j\in\{1,\dots,k\}$, since $P_j(x)=P_j^{(0)}+\mathcal O(|x|)$ as $|x|\to 0$, similarly to the beginning of the proof of this Proposition, we already prove that $\lambda\in\sigma(T)$ considered in $\textrm{\normalfont ran}(P_j)$ if and only if $\lambda^{(0)}\in\sigma\bigl(T^{(0)}\bigr)$ considered in $\textrm{\normalfont ran}\bigl(P_j^{(0)}\bigr)$ where $\lambda^{(0)}$ is the limit of $\lambda$ as $|x|\to 0$. On the other hand, $\textrm{\normalfont ran}\bigl(P_j^{(0)}\bigr)\subset \textrm{\normalfont ran}\bigl(P^{(0)}\bigr)$ due to the fact that $\textrm{\normalfont ran}(P_j)\subset \textrm{\normalfont ran}(P)$. Thus, $\lambda^{(0)}\in\sigma\bigl(T^{(0)}\bigr)$ considered in $\textrm{\normalfont ran}\bigl(P^{(0)}\bigr)$ which is the definition of the eigenvalues $\lambda_j^{(0)}$ for $j\in\{1,\dots,k\}$. Thus, there is a unique $j_0\in\{1,\dots, k\}$ such that $\lambda^{(0)}=\lambda_{j_0}^{(0)}$ since $\lambda_j^{(0)}$ for all $j\in\{1,\dots,k\}$ are distinct. Nonetheless, since $\lambda^{(0)}$ is considered in $\textrm{\normalfont ran}\bigl(P_j^{(0)}\bigr)$, one obtains $j_0=j$ since $P_j^{(0)}P_{j_0}^{(0)}\ne O$ if and only if $j=j_0$. The proof is done.
\end{proof}

\section*{Acknowledgments} The second author's research is supported by the doctoral grant of Gran Sasso Science Institute 2014--2017.

\bibliographystyle{elsarticle-num-sort}
\bibliography{010417}
\end{document}